\documentclass[letterpaper, 10 pt, conference]{ieeeconf}
\IEEEoverridecommandlockouts
\usepackage{amssymb} 
\usepackage{tikz} 
	\usetikzlibrary{calc}
\usepackage{amsmath}
\usepackage{mathtools} 
\usepackage{caption} 
\usepackage{subcaption}

\newtheorem{thm}{Theorem}
\newtheorem{lemma}{Lemma}

\newtheorem{prop}{Proposition}
\newtheorem{conj}{Conjecture}
\newtheorem{defn}{Definition}

\newtheorem{remark}{Remark}
\newtheorem{asmp}{\textbf{Assumption}}

\newcommand{\reals}{\mathbb{R}}

\newcommand{\defeq}{\stackrel{\Delta}{=}}
\newcommand{\norm}[1]{\left \lVert#1\right \rVert}
\newcommand{\snorm}[1]{\left \lVert#1\right \rVert^{2}}
\newcommand{\set}[1]{\{#1\}}
\newcommand{\setc}[2]{\{#1\ |\ #2\}}

\newcommand{\orth}{\perp}
\newcommand{\orthc}{||}

\newcommand{\Id}{\mathrm{Id}}
\newcommand{\bracket}[1]{\left( #1 \right)}

\DeclareMathOperator{\diag}{diag}

\DeclarePairedDelimiterX{\inp}[2]{\langle}{\rangle}{#1\ \middle| \ #2}
\DeclarePairedDelimiterX{\inpl}[2]{\langle}{.}{#1\ \middle| \ #2}
\DeclarePairedDelimiterX{\inpr}[2]{.}{\rangle}{#1\ \ #2}

\tikzstyle{block} = [draw, fill=white, rectangle, 
    minimum height=2em, minimum width=2em]
\tikzstyle{sum} = [draw, fill=white, circle, node distance=0.5cm, inner sep=0pt, minimum size=0.25cm]
\tikzstyle{input} = [coordinate]
\tikzstyle{output} = [coordinate]
\tikzstyle{pinstyle} = [pin edge={to-,thin,black}]
\tikzstyle{branch}=[fill,shape=circle,minimum size=3pt,inner sep=0pt]

\usepackage{multirow}
\usetikzlibrary{cd}

\newtheorem{exmp}{Example}

\title{On the exact convergence to Nash equilibrium in hypomonotone regimes under full and partial-information}

\author{Dian Gadjov and Lacra Pavel
\thanks{This work was supported by an NSERC Discovery Grant.}%
\thanks{\footnotesize D. Gadjov and L. Pavel are with Dept. of Electrical and Computer Engineering, University of Toronto, 
        {\tt\footnotesize dian.gadjov@mail.utoronto.ca},
        {\tt\footnotesize pavel@control.utoronto.ca}}%
}

\begin{document}

\maketitle
\thispagestyle{empty}
\pagestyle{empty}

\begin{abstract}
In this paper, we consider distributed Nash equilibrium seeking in monotone and hypomonotone games. We first assume that each player has knowledge of the opponents' decisions and propose a passivity-based modification of the standard gradient-play dynamics, that we call ``Heavy Anchor''. We prove that Heavy Anchor allows a relaxation of strict monotonicity of the pseudo-gradient, needed for gradient-play dynamics, and can ensure exact asymptotic convergence in merely monotone regimes. We extend these results to the setting where each player has only partial information of the opponents' decisions. Each player maintains a local decision variable and an auxiliary state estimate, and communicates with their neighbours to learn the opponents' actions. We modify Heavy Anchor via a distributed Laplacian feedback and show how we can exploit equilibrium-independent passivity properties to achieve convergence to a Nash equilibrium in hypomonotone regimes.
\end{abstract}

\section{Introduction}

Recent years have seen a flurry of research papers on distributed Nash equilibrium seeking, due to the increase of distributed systems. There are a broad range of networked scenarios that involve strategic interacting agents, where centralized approaches are not suitable. Some examples are demand-side management for smart grids, \cite{Basar2012}, electric vehicles, \cite{PEVParise}, competitive markets, \cite{LiDahleh}, network congestion control, \cite{Garcia}, power control and resource sharing in wireless/wired peer-to-peer networks, cognitive radio systems, \cite{Scutari_2014}. 

Classically, Nash Equilibrium (NE) seeking algorithms assume that each player has knowledge of every other agent's decision/action, the so called \textit{full-decision information} setting. In this setting, there are many well known algorithms that find the NE under various assumptions \cite{BasarLi1987}, \cite{FP07}, \cite{monoBookv2}. In a slightly more general setting, some algorithms require a centralized coordinator that broadcasts data to the network of agents, \cite{gramDR}. In recent years, a collective effort have been made to generalize these results to the \textit{partial-decision information} setting, where a centralized coordinator does not exists. Without a coordinator agents only have partial knowledge of the other agent's action, but may communicate locally with neighboring agents. A variety of NE seeking algorithms, for the partial-decision information setting, have been proposed, e.g. \cite{NedichDistAgg}-\cite{dianAgg}. However, all these results require \textit{strict/strong monotonicity} of the pseudo-gradient. Unfortunately, there are prominent classes of games that do not satisfy this assumption, e.g. in zero-sum games, saddle-point problems, Cournot games,  \cite{ShanbhagTikhonov}, or in resource allocation games,  \cite{Scutari_2014}. 
 
Some existing NE seeking methods are applicable for games with a \textit{merely monotone} pseudo-gradient, but only in \textit{full-decision information} settings, e.g. \cite{FP07}, \cite{monoBookv2}, \cite{frb}. However, these methods typically require more complex computations, such as the forward-backward-forward algorithm, \cite{gramTseng} \cite{monoBookv2}, Tikhonov proximal-point algorithm in \cite{ShanbhagTikhonov}, inexact proximal best-response in \cite{Scutari_2014}, \cite{YiTCNS},  or proximal-point/resolvent computation (e.g.  Douglas-Rachford splitting), \cite{monoBookv2}. Even though proximal-point algorithms or Douglas-Rachford splitting can achieve exact convergence to a NE, they are computationally expensive since each step involves solving an optimization problem. These methods are only applicable in games with easily computed prox (resolvent) operators, \cite{FP07}, \cite{monoBookv2}. Regularization methods, such as the Tikhonov regularization \cite{ShanbhagTikhonov}, or continuous-time mirror-descent dynamics, \cite{boMirror}, are simpler, but require diminishing step-sizes (very slow convergence), or ensure convergence to only an approximate NE. We emphasize that all these existing methods for monotone games, assume agents have perfect knowledge of the actions of the other agents. Additionally, none of these methods deal with hypomonotone games for either the full or partial information setting.

An extremum seeking method \cite{extreme_krstic} for continuous time  monotone games has been proposed. However, this method only converges to an $\epsilon$ neighborhood of the NE. A payoff based method \cite{bandit} for discrete time is recently proposed to find the NE in monotone games. Agents random perturbation their action and moves in the direction of improvement. The random nature of the algorithm with the diminishing step sizes result in a slow method to converge to the NE but at the benefit of just using payoff information.

\textit{Contributions. }
Recognizing the lack of results for (hypo)monotone games, in this paper, we consider NE seeking for games with a \textit{monotone} or \textit{hypomonotone} pseudo-gradient. We propose an algorithm we call ``Heavy Anchor'', constructed by a passivity-based modification of the standard gradient-play dynamics. We demonstrate that in the full-decision information setting, Heavy Anchor ensures exact convergence to a NE for any positive parameter values. Additionally, we show that under a carefully chosen change of coordinates, and conditions on the parameters, Heavy Anchor converges in hypomonotone games. Furthermore, we extend the result to the partial-decision information setting, by using a distributed Laplacian feedback. More specifically, we prove convergence for monotone extended pseudo-gradient, or (hypo)monotone and inverse Lipschitz pseudo-gradient. To the best of our knowledge these are the first such results in the literature. Lastly, we look at quadratic games, an important subclass of games, and derive tighter conditions for the full information (hypo)monotone setting and the partial information (hypo)monotone setting.

Heavy Anchor can be interpreted as modifying the standard gradient method with a term approximating the derivative of the agent's \emph{own} action as predictive term. We use the approximation as a frictional force to improve stability. Heavy Anchor is similar to \cite{antipin}, \cite{shamma} used in saddle-point problems in the full information setting. However, \cite{antipin}, \cite{shamma}, approximate the derivative of the \emph{other} agents' actions. Furthermore, our convergence results are global, unlike the local results in \cite{shamma}. 

In the physics literature, similar dynamics were investigated for stabilizing unknown equilibrium in chaotic systems \cite{physical2}, \cite{physical3}. However, they do not provide a rigorous characterization describing when the equilibrium is stabilized. Finally, Heavy Anchor is also related to second-order dynamics used in the optimization literature, e.g. \cite{attouch1}. If we discretize Heavy Anchor and restrict the parameter values, we can recover the optimistic gradient-descent/ascent (OGDA) \cite{optimistic} or the shadow Douglas Rachford \cite{shadow}. However, all these optimization methods assume that the map is the gradient of a convex function. This does not hold in a game typically -  the game map  a pseudo-gradient rather than a full gradient,  and thus convergence results are not applicable in a game context. Moreover, all these results are for the full information setting. 

In \cite{dianMono} we presented the algorithm and proved convergence for monotone games. In this paper, we extend our results to hypomonotone games and provide additional analysis of inverse Lipschitz operators. Furthermore, we derive tighter conditions for the class of quadratic games, which were not analyzed in \cite{dianMono}.

The paper is organized as follows. Section \ref{sec:background} gives preliminary background. Section \ref{sec:formulation} formulates the problem, standing assumptions and introduces the NE seeking algorithm for the full information case. The convergence analysis is presented in Section \ref{sec:alg}. Section \ref{sec:dist} presents the partial-information version of the algorithm. Section \ref{sec:invLip} investigates a property we are calling ``inverse Lipschitz'', critical in our analysis of the partial information setting and hypomonotone games. Section \ref{sec:partialConv} proves convergence for the partial information setting. Section \ref{sec:Quad} derives tighter conditions for the class of quadratic games. Section \ref{sec:sim} shows simulations of our proposed algorithm and concluding remarks are given in Section \ref{sec:conclusion}.

\emph{Notations}. For $x\in\reals^{n}$, $x^{T}$ denotes its transpose and $\norm{x} = \sqrt{\inp*{x}{x}} = \sqrt{x^{T}x}$ the norm induced by inner product $\inp*{\cdot}{\cdot}$. For a matrix $A\in\reals^{n\times n}$,  $\lambda (A) = \set{\lambda_{1},\dots,\lambda_{n}}$ and $\sigma (A) = \set{\sigma_{1},\dots,\sigma_{n}}$ denotes its eigenvalue and singular value set, respectively. Given $A, B\in \reals^{n\times n}$, let $A \succeq B$ denote that $(A-B)$ is positive semidefinite.
 
For $\mathcal{N}  = \set{1, \dots, N}$, $col(x_{i})_{i\in \mathcal{N}} = [x_{1}^{T},\dots,x_{N}^{T}]^{T}$ denotes the stacked vector of $x_{i}$, while $diag(x_{i})_{i\in\mathcal{N}}$ is the diagonal matrix with $x_{i}$ along the diagonal. $I_{n}$, $\mathbf{1}_{n}$ and $\mathbf{0}_{n}$ denote the identity matrix, the all-ones and the all-zeros vector of dimension $n$,  and  $\otimes$  denotes the Kronecker product. Lastly, we denote $\mathfrak{j} = \sqrt{-1}$.

\section{Background} \label{sec:background}

\subsection{Monotone Operators}
	The following are from \cite{monoBookv2}. Let $T : \mathcal{H}  \to 2^{\mathcal{H}}$ be  an operator, where $\mathcal{H}$ is a Hilbert space. Its graph is denoted by $graT = \setc{(x,y)\in \mathcal{H}\times \mathcal{H}}{y \in Tx}$.
	An operator $T$ is $\mu$-strongly monotone and monotone, respectively, if it satisfies, $\inp*{Tx - Ty}{x-y} \geq \mu \snorm{x-y}$ $\forall x,y \in \mathcal{H}$, where $\mu >0$ and $\mu = 0$, respectively. Additionally, we say an operator $T$ is $\mu$-hypomonotone if $\inp*{Tx - Ty}{x-y} \geq -\mu \snorm{x-y}$ $\forall x,y \in \mathcal{H}$, where $\mu \geq 0$.
	 $T$ is maximally monotone if $\forall (x,y)\in \mathcal{H}\times \mathcal{H}$, $(x,y)\in graT  \iff  (\forall (u,v)\in graT)\inp*{x-u}{y-v} \geq 0$. The resolvent of a monotone operator $T$ is denoted by $\mathcal{J}_{\lambda T} = \bracket{\text{Id} + \lambda T}^{-1}$, $\lambda>0$, where $\text{Id}$ is the identity operator. Fixed points of $\mathcal{J}_{\lambda T} $ are identical to zeros of $T$ (Prop. 23.2,  \cite{monoBookv2}). 
An operator $T$ is $L$-Lipschitz if, $\norm{Tx- Ty} \leq L \norm{x-y}$ $\forall x,y \in \mathcal{H}$. An operator $T$ is $C$-cocoercive ($C$-inverse strongly monotone) if	$\inp*{Tx - Ty}{x-y} \geq C \snorm{Tx-Ty}$ $\forall x,y \in \mathcal{H}$.

\subsection{Equilibrium Independent Passivity}
The following are from \cite{EIP}. Consider a system,
\begin{align} \label{eqn:generalDynSys}
\begin{split}
	\dot{x} &= f(x,u) \\
	y &= h(x,u)
\end{split}
\end{align}
with $x\in \reals^n$,  $u \in \reals^q$ and $y \in \reals^q$, $f$ locally Lipschitz and $h$ continuous. For a differentiable  function $V:\reals^n \rightarrow \reals$, the time derivative of $V$ along solutions of \eqref{eqn:generalDynSys} is denoted by $\dot{V}(x) = \nabla^T V(x) \, f(x,u)$ or just $\dot{V}$. Let  $\overline{u}$, $\overline{x}$, $\overline{y}$ be  an equilibrium condition, such that $0=f(\overline{x},\overline{u})$, $\overline{y}=h(\overline{x},\overline{u})$. 
Equilibrium independent passivity (EIP) requires a system to be passive independent of the equilibrium point. 
\begin{defn}\label{def:EIP}
System \eqref{eqn:generalDynSys} is Equilibrium Independent Passive (EIP) if it is passive with respect to $\overline{u}$ and $\overline{y}$; that is for every $ \overline{u}  \in \overline{U}$ there exists a differentiable, positive semi-definite storage function $V: \reals^n \to \reals$ such that $V(\overline{x}) = 0$ and  $\forall u \in \reals^q$, $x \in \reals^n$, $\dot{V}(x)  \leq   \inp*{y-\overline{y}}{u-\overline{u}}$. The system is Output-strictly EIP if, $\dot{V}  \leq  \inp*{y-\bar{y}}{u - \bar{u}} - \rho \snorm{y-\bar{y}}$ where $\rho > 0$. 
\end{defn}

\subsection{Graph Theory}
Let the graph $G = (\mathcal{N}, \mathcal{E})$ describe the information exchange among a set $\mathcal{N}$ of agents, where $\mathcal{E} \subset \mathcal{N} \times \mathcal{N}$. 
If agent $i$ can get information from agent $j$, then $(j, i) \in \mathcal{E}$ and agent $j$ is in agent $i$'s neighbour set $\mathcal{N}_{i} = \setc{j}{(j, i)\in \mathcal{E}}$. $G$ is undirected when $(i, j) \in  \mathcal{E}$ if and only if $(j, i)  \in  \mathcal{E}$. $G$ is connected if there is a path between any two nodes. Let $W = [w_{ij}] \in  \reals^{N\times N}$ be the weighted adjacency matrix, with $w_{ij} > 0$ if $j \in  \mathcal{N}_{i}$ and $w_{ij} = 0$ otherwise. Let $Deg = \diag(d_{i})_{i\in\mathcal{N}}$, where $d_{i} = \sum_{j=1}^{N} w_{ij}$. Assume that $W = W^{T}$ so the weighted Laplacian of $G$ is $L = Deg - W$. When $G$ is connected and undirected, $0$ is a simple eigenvalue of $L$, $L\mathbf{1}_N = \mathbf{0}$, $\mathbf{1}_N^{T}L = \mathbf{0}^{T}$, and all other eigenvalues are positive, 
$0 < \lambda_{2}(L) \leq \dots  \leq  \lambda_{N}(L)$.

\section{Problem Setup}\label{sec:formulation}

Consider a set $\mathcal{N}=\{ 1,\dots,N\}$ of $N$ players (agents) involved in a game. Each player $i \in \mathcal{N}$ controls its action or decision $x_i \in \Omega_i \subseteq  \reals^{n_i}$. The action set of all players is the Cartesian product $\Omega = \prod_{i\in\mathcal{N}}\Omega_i \subseteq \reals^{n}$, $n = \sum_{i\in\mathcal{N}} n_i$. Let $x=(x_i,x_{-i})\in  {\Omega}$ denote all agents' action profile or $N$-tuple,  where  $x_{-i}$ 
is the $(N-1)$-tuple of all agents' actions except agent $i$'s. Alternatively, $x$ is represented as a stacked vector $x = [x_1^T \dots x_N^T]^T  \in \Omega \subseteq  \reals^n$. Each player (agent) $i$ aims to minimize its  own cost function $J_i(x_i,x_{-i})$, $J_i : \Omega \to \reals$, which depends on possibly all other players' actions. 
 Let the game thus defined be denoted by $\mathcal{G}(\mathcal{N},J_i,\Omega_i)$.
\begin{defn}\label{defNE}
	Given a game $\mathcal{G}(\mathcal{N},J_i,\Omega_i)$, an action profile $x^* =(x_i^*,x_{-i}^*)\in \Omega$ is  a Nash Equilibrium (NE) of $\mathcal{G}$  if
	\begin{align*}
		(\forall i \in \mathcal{N})(\forall y_i \in \Omega_i) \quad J_i(x_i^*,x_{-i}^*) \leq J_i(y_i,x_{-i}^*)
	\end{align*}
and therefore no agent has the incentive to unilaterally deviate from their action.
\end{defn}

Alternatively, if $J_{i}$ is differentiable then a NE $x^* \in  \Omega$ satisfies the variational inequality (VI) (Proposition 1.4.2, \cite{FP07}),
\begin{align} \label{eq:ViNash}
	(x-x^*)^TF(x^*)\geq 0 \quad \forall x\in {\Omega}
\end{align}
where $F : \Omega  \to  \reals^n$ is the \emph{pseudo-gradient (game) map} defined by stacking all agents'  partial gradients, 
\begin{align}\label{eq:expPsuedoGrad_F}
F(x) = [\nabla_{x_1} J^T_1(x),\dotsc,\nabla_{x_N} J^T_N(x)]^T
\end{align}
with $\nabla_{x_i} J_i(x_i,x_{-i}) =\frac{\partial J_i}{\partial x_i}(x_i,x_{-i})\in  \reals^{n_i}$, the partial-gradient of $J_i(x_i,x_{-i})$ with respect to its own action $x_i$.

We use the following basic convexity and smoothness assumption, which ensures the existence of a NE.

\begin{asmp} \label{asmp:Jsmooth}
For every $i\in\mathcal{N}$, $\Omega_i=\reals^{n_i}$ and the cost function  $J_i:\Omega  \to  \reals$  is $\mathcal{C}^1$ in its arguments, convex and radially unbounded in $x_i$,  for every $x_{-i}\in {\Omega}_{-i}$.
\end{asmp}
Under Assumption \ref{asmp:Jsmooth} from Corollary 4.2 in  \cite{B99} it follows that a NE $x^*$ exists. Furthermore, the VI \eqref{eq:ViNash} reduces to $ F(x^{*})=0$.

A standard method for reaching a Nash Equilibrium (NE) is using gradient-play dynamics \cite{flam}, i.e., 
\begin{align} \label{eqn:gradientDescent}
 \dot{x}_{i} &= -\nabla_{x_i}J_{i}(x_{i},x_{-i}), \forall i\in \mathcal{N}, \quad \text{or} \quad 	\dot{x} = -F(x)
\end{align}
This algorithm converges to the NE if the pseudo-gradient is strictly monotone but may fail if the pseudo-gradient is only monotone. For example, consider a $2$-player zero-sum game where the cost functions are $J_{1}(x_{1},x_{2}) = x_{1}x_{2}$, $J_{2}(x_{1},x_{2}) = -x_{1}x_{2}$. The pseudo-gradient is,
\begin{align}
	F(x) &= \begin{bmatrix}
		0 & 1 \\
		-1 & 0
	\end{bmatrix}x \label{eqn:harmonic}
\end{align}
which is monotone and the NE is $(0,0)$. If the initial state $x(0) \neq (0,0)$ then \eqref{eqn:gradientDescent} will cycle around the NE and never converge, i.e., Figure \ref{fig:grad_harmonic}. In this paper, we are interested in monotone games.

\begin{figure}[ht]
\centering
\begin{minipage}[ht]{1\columnwidth}
	\centerline{\includegraphics[width=6cm]{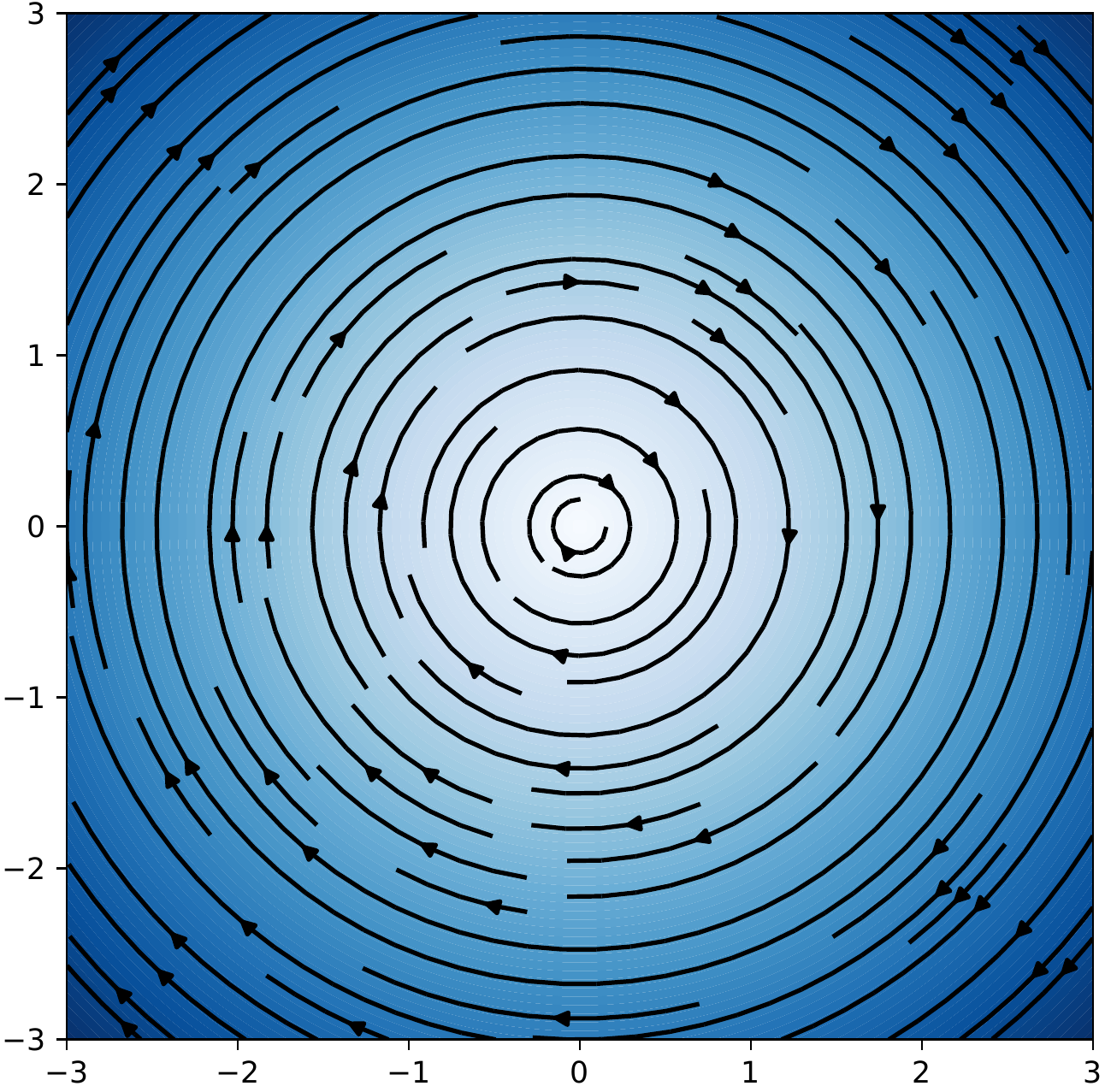}}
	\caption{Gradient vector field}\label{fig:grad_harmonic}
\end{minipage}
\end{figure}

\begin{asmp} \label{asmp:monotone}
	The pseudo-gradient $F$ is monotone.
\end{asmp}
Under Assumption \ref{asmp:Jsmooth} and \ref{asmp:monotone}, the set of NE  is  convex, (cf. Theorem 3, \cite{Scutari_2014}), characterized by $\{x^*| F(x^{*})=0\}$.

\subsection{Proposed Algorithm}

The dynamics \eqref{eqn:gradientDescent} can be viewed as an open-loop system with no feedback. We propose a new algorithm, what we are calling ``Heavy Anchor'', by modifying the feedback path with a bank of high-pass filters as depicted in Figure \ref{fig:bdAlg} below, with $u_{c}=\mathbf{0}$. We call it Heavy Anchor because we show that it looks like Polyak's heavy ball method but with the momentum term having the opposite sign.

\begin{figure}[ht]
\begin{center}
      \begin{tikzpicture}[scale=0.6]
        \def\width{5} \def\height{3} 
        \tikzstyle{block} = [draw, fill=white, rectangle,minimum height=10mm, minimum width=15mm, line width=0.5mm] 
        \tikzstyle{sum} = [draw, thick, fill=white, circle, inner sep=0,minimum size=0.6, line width=0.5mm]

        \node[block] (C) at (0,0) {$\begin{aligned}
            \dot{r} &= -\alpha r + \alpha u_{2} \\
            y_{2} &= -\beta r + \beta u_{2}
          \end{aligned}$}; 
        \node[block] (Sys) at ($(C)+(0,\height)$) {$	\begin{aligned}
            \dot{x} &= -F(x) + u_{1} \\
            y_{1} &= x
          \end{aligned}$}; 
        \node[sum] (fb) at ($(Sys) - (\width,0)$) {\Large $+$};

		\coordinate (TR) at ($(Sys) + (\width,0)$); 
		\coordinate (BL) at ($(C) - (\width,0)$); 
		\coordinate (BR) at ($(C) + (\width,0)$);
		\coordinate (control) at ($(fb) - 0.5*(\width,0)$) {};
        \coordinate (out) at ($(TR) + 0.25*(\width,0)$) {};

		\path[-,line width=0.5mm] 
		(Sys) edge node[pos=0.5, anchor=south] {$y_{1}$} (TR) (BR) edge node[pos=0.5, anchor=south] {$u_{2}$} (C)
		(C) edge node[pos=0.5, anchor=south] {$y_{2}$} (BL)
        (TR) edge (BR); 
        \path[->,line width=0.5mm] 
        (BL) edge node[pos=0.7, anchor=east] {$-$} (fb)
        (fb) edge node[pos=0.5, anchor=south] {$u_{1}$} (Sys)
        (control) edge node[pos=0.5, anchor=south] {$u_{c}$} (fb)
        (TR) edge (out)
        ;
      \end{tikzpicture}
\end{center}
\caption{Block diagram of \eqref{eqn:proposedAlg}} \label{fig:bdAlg}
\end{figure}
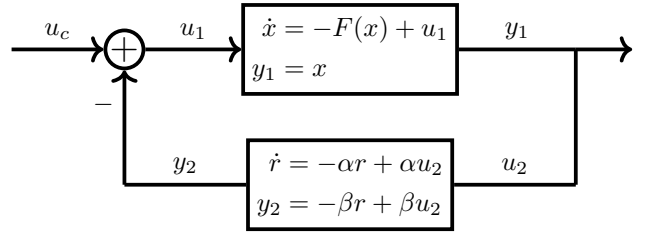

Explicitly the dynamics are,
\begin{align} \label{eqn:proposedAlg}
\begin{split}
	\dot{r} &= \alpha (x-r) \\	
	\dot{x} &= -F(x) - \beta(x-r) \\
\end{split} \tag{HA$_{\text{F}}$}
\end{align}
where $\alpha, \beta \in \reals_{++}$ and $r\in\reals^{n}$ are auxiliary variables. The individual agent dynamics are,
\begin{align*}
	 \dot{r}_{i} &= \alpha(x_{i}-r_{i}) \\
	 \dot{x}_{i} &= -\nabla_{x_i}J_{i}(x_{i},x_{-i}) - \beta(x_{i}-r_{i})
\end{align*}
The new dynamics have a gradient-play component with a dynamic estimation of the own action derivative. Figure \ref{fig:anchor_harmonic} shows the decision trajectories $x$ under Heavy Anchor for the two player zero-sum game \eqref{eqn:harmonic}.

\begin{figure}[ht]
\centering
\begin{minipage}[ht]{1\columnwidth}
	\centerline{\includegraphics[width=6cm]{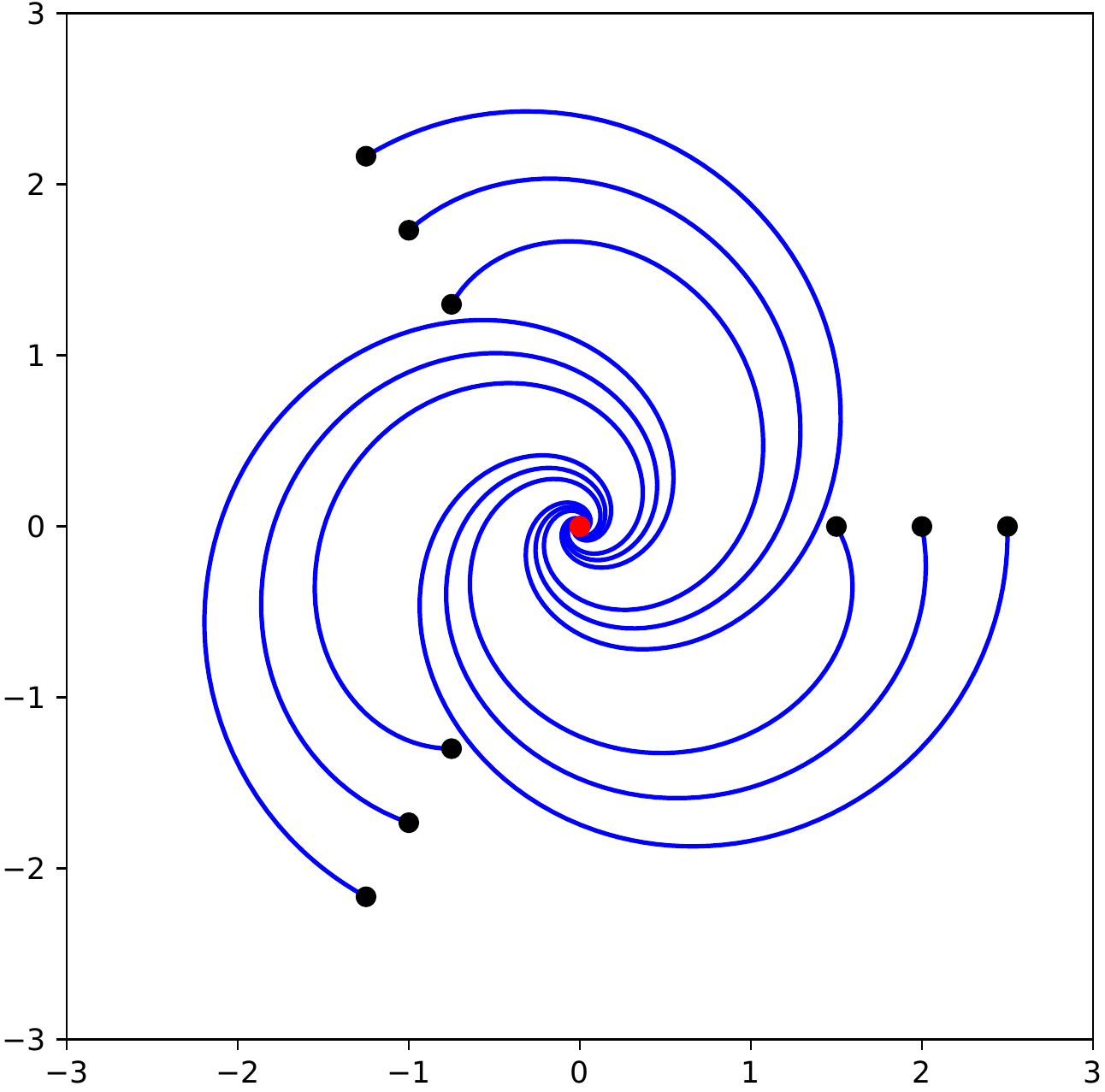}}
	\caption{Decision trajectories under Heavy Anchor}\label{fig:anchor_harmonic}
\end{minipage}
\end{figure}

\subsection{Connections to Other Dynamics/Algorithms}

Our proposed dynamics \eqref{eqn:proposedAlg} is related to other continuous-time dynamics or discrete-time algorithms. First, \eqref{eqn:proposedAlg} can be written as the second-order dynamics,
\begin{align}\label{eqn:sec_proposedAlg} 
	\ddot{x} + (\nabla F(x) + \beta + \alpha)\dot{x} + \alpha F(x) = 0.
\end{align}
Under appropriate 
restrictions on the values of $\alpha$ and $\beta$, this dynamics recovers other dynamics/algorithms. For example,   similar dynamics appears in stabilizing unknown equilibrium in chaotic systems or saddle functions  \cite{physical2}, \cite{physical3}. However, these works do not rigorously characterize stability/convergence. As another example, consider

\begin{align*}
	\ddot{x} + \alpha \dot{x} + \beta \nabla^{2}f(x)\dot{x} + \nabla f(x) + \nabla \Psi(x) = 0,
\end{align*}
where $f$ is a convex function, as considered in the optimization literature,  \cite{attouch1}, \cite{radu}. 
If $F=\nabla f$,  \eqref{eqn:sec_proposedAlg} can be written as the above (with $\Psi(x)
 \equiv  0$). However,  in a game $F$ is not a true gradient (unless the game is a potential game), but rather a pseudo-gradient, so convergence results are not applicable. 

Next, we  relate \eqref{eqn:proposedAlg} to some existing discrete-time algorithms. 
Performing an Euler discretization of \eqref{eqn:proposedAlg} gives,
\begin{align*}
	x_{k+1} &= x_{k} - sF(x_{k}) - s\beta(x_{k}-r_{k}) \\
	r_{k+1} &= r_{k} + s\alpha (x_{k} - r_{k})
\end{align*}
where $s > 0$ is the step size, which after some manipulations yields  the second-order difference equation,
\begin{align} \label{eqn:secondOrder}
x_{k+2} &= x_{k+1} - \alpha s^{2}F(x_{k+1}) 
	+ (1 - s\alpha -s\beta)\bracket{x_{k+1}-x_{k}} \notag\\
	&\quad - s(1-s\alpha)\bracket{F(x_{k+1})-F(x_{k})}.	
\end{align} 

Depending on how the parameters $\alpha$ and $\beta$ are selected we can recover some known algorithms. 	
	If $\alpha = \beta = \frac{1}{2s}$ and $F = \nabla f$ for some convex function $f$ then \eqref{eqn:secondOrder} becomes,
\begin{align*}
x_{k+2} &= x_{k+1} - \tilde{s}\bracket{2\nabla f(x_{k+1})-\nabla f(x_{k})}
\end{align*}
	where $\tilde{s} = \frac{s}{2}$ gives the optimistic gradient-descent/ascent (OGDA) \cite{optimistic}, shadow Douglas Rachford \cite{shadow}, or the forward-reflected backward method \cite{frb}. On the other hand,  if $\alpha = \frac{1}{s}$, and $F = \nabla f$ then \eqref{eqn:secondOrder} becomes,
\begin{align*}
	x_{k+2} &= x_{k+1} - s\bracket{\nabla f(x_{k+1}) + \beta \bracket{x_{k+1}-x_{k}}}
\end{align*}
where $\beta < 0$ gives Polyak's heavy-ball method, \cite{polyakBall}.

\section{Convergence under Perfect Information} \label{sec:alg}

In this section we  consider that each agent knows all $x_{-i}$ (actions that his cost depends on), hence the full (perfect) decision information setting. 
In Theorem \ref{thm:EIPconvergence} we show  that the continuous-time dynamics \eqref{eqn:sec_proposedAlg} converges for all $\alpha, \beta  > 0$, in this full  information setting. Our idea is to see \eqref{eqn:proposedAlg} as an (EIP) passivity-based feedback modification of \eqref{eqn:gradientDescent}.
To prove that $x$ in \eqref{eqn:proposedAlg} converges to an Nash Equilibrium in monotone games, we decompose the system into a feedback interconnection between two subsystems (see Fig. \ref{fig:bdAlg}). 

We show that each subsystem is EIP and use their storage functions to construct an appropriate Lyapunov function to prove that the equilibrium point of the interconnected system (which is a NE) is asymptotically stable. 

\begin{lemma} \label{lemma:eipSub1}
	Under Assumption \ref{asmp:monotone} the following system,
\begin{align} \label{eqn:eip1}
\begin{split}
	\dot{x} &= -F(x) + u_{1} \\
    y_{1} &= x.
\end{split}
\end{align}
is EIP with repect to $u_{1}$ and $y_{1}$.
\end{lemma}
\begin{proof}
	Let $\bar{x}$ be the equilibrium of \eqref{eqn:eip1} for input $\bar{u}_{1}$, and $\bar{y}_{1}$ the corresponding output. Consider the storage function $V_{1}(x) = \frac{1}{2}\snorm{x-\bar{x}}$. Then, along the solutions of \eqref{eqn:eip1},
\begin{align}
	\dot{V}_{1}(x) &= \inp*{x-\bar{x}}{-F(x)+u_{1}+F(\bar{x})-\bar{u}_{1}} \notag \\
	&= -\inp*{x-\bar{x}}{F(x)-F(\bar{x})} + \inp*{u_{1} - \bar{u}_{1}}{y_{1} - \bar{y}_{1}} \label{eqn:EIPinequality1}
\end{align}
By Assumption \ref{asmp:monotone}, the first term is $\leq0$ and the system is EIP.
\end{proof}

\begin{lemma} \label{lemma:eipSub2}
	For any $\alpha, \beta > 0$ the following system,
\begin{align} \label{eqn:eip2}
\begin{split}
	\dot{r} &= -\alpha r + \alpha u_{2} \\
	y_{2} &= -\beta r + \beta u_{2}.
\end{split}
\end{align}
is OSEIP with repect to $u_{2}$ and $y_{2}$.
\end{lemma}
\begin{proof}
	Let $\bar{r}$ be the equilibrium of \eqref{eqn:eip2} for the input $\bar{u}_{2}$ and let $\bar{y}_{2}$ be the corresponding output. Consider the storage function $V_{2}(r) = \frac{\beta}{2\alpha}\snorm{r-\bar{r}}$. Then, along solutions of \eqref{eqn:eip2},
	\begin{align}
		\dot{V}_{2}(r) &= \frac{\beta}{\alpha}\inp*{r-\bar{r}}{-\alpha r + \alpha u_{2} + \alpha \bar{r} - \alpha \bar{u}_{2}} \notag \\
		&= \frac{\beta}{\alpha}\inp*{u_{2} - \frac{1}{\beta}y_{2} - \bar{u}_{2} + \frac{1}{\beta}\bar{y}_{2}}{\frac{\alpha}{\beta}(y_{2}-\bar{y}_{2})} \notag \\
		&= \inp*{u_{2}-\bar{u}_{2}}{y_{2}-\bar{y}_{2}} - \frac{1}{\beta}\snorm{y_{2}-\bar{y}_{2}} \label{eqn:EIPinequality2}
	\end{align}
	Therefore, the system is OSEIP for any $\beta>0$.\end{proof}

We now turn to the interconnected system \eqref{eqn:proposedAlg}. We show first that any equilibrium of \eqref{eqn:proposedAlg} is  a NE. Then, using the two storage functions from Lemma \ref{lemma:eipSub1} and \ref{lemma:eipSub2} we show that any equilibrium point \eqref{eqn:proposedAlg} is asymptotically stable. 
\begin{lemma}\label{eq_FI}
	Any equilibrium of \eqref{eqn:proposedAlg} is $(x^{*},x^{*})$ where $x^{*}$ is a Nash equilibrium of the game. 
\end{lemma}
\begin{proof}
	Let the equilibrium point of \eqref{eqn:proposedAlg} be denoted $(\bar{x},\bar{r})$. Then $0 = \alpha (\bar{x}-\bar{r})$ implies that $\bar{x} = \bar{r}$ and $0 = -F(\bar{x}) - \beta(\bar{x} - \bar{r}) = -F(\bar{x})$. An equilibrium $x^{*}$ of \eqref{eqn:gradientDescent} is such that $F(x^*) = 0$ therefore $\bar{x} = \bar{r} = x^{*}$ a NE.
\end{proof}

\begin{thm}\label{thm:EIPconvergence}
Consider a game $\mathcal{G}(\mathcal{N},J_i,\Omega_i)$ under Assumption \ref{asmp:Jsmooth} and \ref{asmp:monotone}. Let the overall dynamics of the agents  be given by \eqref{eqn:proposedAlg}. Then, for any $\alpha,\beta > 0$, the set of Nash equilibrium points $\setc{(x^*}{F(x^*) = 0}$  is globally asymptotically stable.
\end{thm}
\begin{proof}
Note that  \eqref{eqn:proposedAlg} is the system in Fig. \ref{fig:bdAlg}) with $u_c=0$. Consider the following candidate Lyapunov function $V(x,r) = V_{1}(x) + V_{2}(r)$ where $V_{1}(x) = \frac{1}{2}\snorm{x-\bar{x}}$ and $V_{2}(r)=\frac{\beta}{2\alpha}\snorm{r-\bar{r}}$, where cf. Lemma \ref{eq_FI}, $\bar{x}=\bar{r}=x^*$. Along the solutions of \eqref{eqn:proposedAlg}, from Lemma \ref{lemma:eipSub1}, \eqref{eqn:EIPinequality1}, and Lemma \ref{lemma:eipSub2}, \eqref{eqn:EIPinequality2}, using $u_1=-y_2$, $u_2=y_1$, $\bar{x}=\bar{r}$ and cancelling terms, we obtain, 
\begin{align}\label{EIP_fdb_F}
	\dot{V}(x,r) &= -\inp*{x-\bar{x}}{F(x)-F(\bar{x})} - \frac{1}{\beta}\snorm{y_{2}-\bar{y}_{2}} \notag \\
	&= -\inp*{x-\bar{x}}{F(x)-F(\bar{x})} - \beta \snorm{x-r} 
\end{align}
By Assumption \ref{asmp:monotone}, it follows that $\dot{V} \leq  0$. We resort to LaSalle's Invariance Principle \cite{nonlinear}. Note that $\dot{V} = 0$ implies $x - r = 0$. On $x=r$  the dynamics  \eqref{eqn:proposedAlg} reduces to,
$
	0=\dot{x} - \dot{r} = -F(x) - \beta(x-r) - \alpha(x-r) = -F(x)
$,
hence the largest invariant set is $\setc{x}{F(x)=0}$. Since $V$ is radially unbounded, the conclusion follows. 
\end{proof}

\section{Partial Information} \label{sec:dist}
In Section \ref{sec:alg} we considered that each agent knows all others' decisions $x_{-i}$. In this section we propose a version of \eqref{eqn:proposedAlg}, that works in the partial information setting, i.e. when agents do not know all others' decisions and instead estimate them based on communicating with their neighbors over a communication graph $G_{c}$. 
\begin{asmp} \label{asmp:graph}
	$G_{c} = (\mathcal{N}, \mathcal{E})$ is undirected and connected.
\end{asmp}

Assume that each agent $i$ maintains an estimate vector $\mathbf{x}^{i} = col(\mathbf{x}^{i}_{j})_{j\in\mathcal{N}} \in  \reals^{n}$ where $\mathbf{x}^{i}_{j}$ is agent $i$'s estimate of player $j$'s action. Note that $\mathbf{x}^{i}_{i} = x_{i}$ is player $i$'s actual action. Let $\mathbf{x} = col(\mathbf{x}^{i})_{i\in\mathcal{N}}\in\reals^{Nn}$ represent all agents' estimates stacked into a single vector. Similarly, define the auxiliary variable $\mathbf{r}^{i} \in \reals^{n}$ for each agent $i$. 
Let the extended pseudo-gradient be denoted as $\mathbf{F}(\mathbf{x}) := col(\nabla_{x_i}J_i(\mathbf{x}^{i}))_{i\in\mathcal{N}}$, where each agent uses its estimate of others' decisions instead of true decisions. Note that at consensus of estimates, $\mathbf{x}^{i}=x$, for all $i\in\mathcal{N}$,  and  $\mathbf{F}(\mathbf{1}_{N}\otimes x)=F(x)$, for any $x  \in  \reals^n$. Let the matrix $\mathcal{R} = diag(\mathcal{R}_{i})_{i\in\mathcal{N}}$, where   $\mathcal{R}_{i} = [\mathbf{0}_{n_{i}\times n<i}I_{n_{i}}\mathbf{0}_{n_{i} \times n>i}]$, and $n<i = \sum_{\substack{j<i }}n_{j}$, $n<i = \sum_{\substack{j>i}}n_{j}$, $i,j\in\mathcal{N}$. The matrix $\mathcal{R}_{i}$ is used to get the component of a vector that belongs to agent $i$, i.e., $x_{i} = \mathcal{R}_{i}\mathbf{x}^{i}$ and $x = \mathcal{R}\mathbf{x}$. The operation $\mathbf{x} = \mathcal{R}^{T}x$ sets $\mathbf{x}^{i}_{i}=x_{i}$ and $\mathbf{x}^{i}_{j} = 0$ for all $j\neq i$. 

The problem is thus lifted into an augmented space of decisions, estimates and auxiliary variables $(\mathbf{x},\mathbf{r})$, with the original space being its consensus subspace. Consider the partial information version of \eqref{eqn:proposedAlg}, over $G_{c}$, where 
the individual agent dynamics is given as,
\begin{align} \label{eqn:distProposedAlgAgent}
	\mathbf{\dot{r}}^{i} &= \alpha (\mathbf{x}^{i} - \mathbf{r}^{i}) \\
	\mathbf{\dot{x}}^{i} &= -\mathcal{R}_{i}^{T}\nabla_{x_i}J_{i}(\mathbf{x}^{i}) - \beta(\mathbf{x}^{i} - \mathbf{r}^{i}) - c\sum_{j\in \mathcal{N}_{i}} w_{ij}(\mathbf{x}^{i} - \mathbf{x}^{j}) \notag
\end{align}
or, in compact (stacked) form, as 
\begin{align} \label{eqn:distProposedAlg}
\begin{split}
	\mathbf{\dot{r}} &= \alpha (\mathbf{x} - \mathbf{r}) \\
	\mathbf{\dot{x}} &= -\mathcal{R}^{T}\mathbf{F}(\mathbf{x}) - \beta(\mathbf{x} - \mathbf{r}) - c\mathbf{L}\mathbf{x}
\end{split} \tag{HA$_{\text{\textbf{F}}}$}
\end{align}
where $c> 0$ is a scaling factor and $\mathbf{L} = L\otimes I_{n}$.  The individual agent dynamics \eqref{eqn:distProposedAlgAgent} is the augmented version of  \eqref{eqn:proposedAlg} with a Laplacian (consensus) correction for the estimates. Note that the dynamics \eqref{eqn:distProposedAlg} is similar to  Fig. \ref{fig:bdAlg}, but with an augmented state $(\mathbf{r},\mathbf{x})$, and with feedback loop closed with $u_c = - c\mathbf{L}\mathbf{x}$.  At consensus, $\mathbf{x}=\mathbf{1}_N\otimes x$,  $\mathbf{r}=\mathbf{1}_N\otimes r$, and  \eqref{eqn:distProposedAlg} recovers   \eqref{eqn:proposedAlg}. First we show that any equilibrium point of \eqref{eqn:distProposedAlg} is a NE. 

\begin{lemma} \label{lemma:distEQ}
	Consider a game $\mathcal{G}(\mathcal{N},J_i,\Omega_i)$ under Assumption \ref{asmp:Jsmooth}, over a communication graph $G_c = (\mathcal{N},\mathcal{E})$ satisfying Assumption \ref{asmp:graph}. Let each agents' dynamics be as in \eqref{eqn:distProposedAlgAgent} or overall as \eqref{eqn:distProposedAlg}. Then, any equilibrium $(\bar{\mathbf{x}},\bar{\mathbf{r}})$ of \eqref{eqn:distProposedAlg} satisfies $\bar{\mathbf{x}}^{1} = \cdots = \bar{\mathbf{x}}^{N} = \bar{\mathbf{r}}^{1} =  \cdots = \bar{\mathbf{r}}^{N} = x^{*}$ where $x^{*}$ is a NE.
\end{lemma}
\begin{proof}
	Let $(\bar{\mathbf{x}},\bar{\mathbf{r}})$ denote an equilibrium of \eqref{eqn:distProposedAlg}. Then at equilibrium we have $\bar{\mathbf{x}} = \bar{\mathbf{r}}$ and $\mathbf{0}_{Nn} = -\mathcal{R}^{T}\mathbf{F}(\bar{\mathbf{x}}) - \mathbf{L}\bar{\mathbf{x}}$. Pre-multiplying both sides by $(\mathbf{1}^{T}_{N}\otimes I_{n})$ yields $\mathbf{0}_{n} = \mathbf{F}(\bar{\mathbf{x}})$ and therefore, $\mathbf{0}_{Nn} = -\mathbf{L}\bar{\mathbf{x}}$. By Assumption \ref{asmp:graph}, $\mathbf{0}_{Nn} = -\mathbf{L}\bar{\mathbf{x}}$ when $\bar{\mathbf{x}}^{1} =\cdots = \bar{\mathbf{x}}^{N}$ i.e., $\bar{\mathbf{x}} = \mathbf{1}_{N}\otimes \bar{x}$ for some $\bar{x} \in \reals^{n}$. Therefore, $\mathbf{0}_{n} = \mathbf{F}(\bar{\mathbf{x}}) = \mathbf{F}(\mathbf{1}_{N}\otimes \bar{x}) = F(\bar{x})$ hence, $ \bar{x} = x^{*}$, 	
	where $x^{*}$ is a Nash Equilibrium. 
\end{proof}

\begin{remark} We note that in the full decision information case, monotonicity of $F$ was instrumental (see Theorem \ref{thm:EIPconvergence}). In the augmented space monotonicity does not necessarily hold, even if on the consensus subspace it does cf. Assumption \ref{asmp:monotone}, see \cite{PavelGNE}. This is unlike distributed optimization, where due to separability,  the extension of monotonicity/convexity properties to the augmented space holds automatically. This is the main technical difficulty in developing NE seeking dynamics in partial-information settings. \end{remark}

Our first result is proved under a monotonicity assumption on the extended pseudo-gradient $\mathbf{F}$. 

\begin{asmp}\label{asmp:extendMono}
	The extended pseudo-gradient is monotone, $\inp*{\mathbf{x}-\mathbf{x'}}{\mathcal{R}^{T}(\mathbf{F}(\mathbf{x}) - \mathbf{F}(\mathbf{x'}))} \geq 0$, $\forall \mathbf{x}, \mathbf{x'}$. 
\end{asmp}
Assumption \ref{asmp:extendMono}  has been also used in Thm.1, \cite{dianCT}, or \cite{Hu} (as cocoercivity). It represents extension of monotonicity off the consensus subspace.  Note that on the consensus subspace ($\mathbf{x}=\mathbf{1}_N\otimes x$), it is automatically satisfied by   Assumption \ref{asmp:monotone}. Under Assumption  \ref{asmp:extendMono}, the following result can be immediately obtained by exploiting EIP properties. 
\begin{thm}
	Consider a game $\mathcal{G}(\mathcal{N},J_i,\Omega_i)$ under Assumption \ref{asmp:Jsmooth}, \ref{asmp:monotone}, and \ref{asmp:extendMono}, over a communication graph $G_{c} = (\mathcal{N},\mathcal{E})$ satisfying Assumption \ref{asmp:graph}. Let the overall dynamics of the agents be given by \eqref{eqn:distProposedAlg} or \eqref{eqn:distProposedAlgAgent}. Then, for any $\alpha,\beta > 0$,   \eqref{eqn:distProposedAlg} converges asymptotically to $(\mathbf{1}_{N}\otimes x^*,\mathbf{1}_{N}\otimes x^*)$ where $x^*$ is a NE.
	\end{thm}
\begin{proof}
Note that \eqref{eqn:distProposedAlg} is similar to a dynamics as in Fig. \ref{fig:bdAlg}, but with an augmented state $\mathbf{x}$ (decisions and estimates), and with feedback loop closed with $u_c = - c\mathbf{L}\mathbf{x}$. We exploit the EIP properties of the two, forward and feedback, subsystems. Namely, consider 
$V(\mathbf{x},\mathbf{r}) = \frac{1}{2}\snorm{\mathbf{x}-\bar{\mathbf{x}}} + \frac{\beta}{2\alpha}\snorm{\mathbf{r}-\bar{\mathbf{r}}}$, where $\bar{\mathbf{x}} = \bar{\mathbf{r}}= \mathbf{1}_{N}\otimes \bar{x}$ (cf. Lemma \ref{lemma:distEQ}). Then, along solutions of \eqref{eqn:distProposedAlg}, similar to \eqref{EIP_fdb_F} in Theorem \ref{thm:EIPconvergence}, we can obtain, 
\begin{align}\label{EIP_fdb_boldF}
	\dot{V}(\mathbf{x},\mathbf{r}) &= -\inp*{\mathbf{x}-\bar{\mathbf{x}}}{\mathcal{R}^{T}\mathbf{F}(\mathbf{x})-\mathcal{R}^{T}\mathbf{F}(\bar{\mathbf{x}})} \notag \\
	& \quad 
	+\inp*{\mathbf{x}-\bar{\mathbf{x}}}{u_c-\bar{u}_c} - \beta \snorm{\mathbf{x}-\mathbf{r}} 
\end{align}
where  $u_{c} = -c\mathbf{L}\mathbf{x}$. The first term is nonpositive under Assumption \ref{asmp:extendMono}. For any $\alpha,\beta
>0$,  the system is strictly EIP from $u_{c}$ to $y_{1}=\mathbf{x}$, and with $u_{c} = -c\mathbf{L}\mathbf{x}$, since $\mathbf{L}$ is positive semidefinite, it follows that $\dot{V}  \leq 0$. We use LaSalle's Invariance Principle and  find the largest invariant set \cite{nonlinear}. Note that  $\dot{V} = 0$ implies that $\mathbf{x}=\mathbf{r}$ and $\mathbf{L}\mathbf{x}=\mathbf{L}\bar{\mathbf{x}}$. Since $\bar{\mathbf{x}} = \mathbf{1}_{N}\otimes \bar{x}$ (cf. Lemma \ref{lemma:distEQ}), $\mathbf{L}\mathbf{x}=\mathbf{L}\bar{\mathbf{x}}=0$, hence $\mathbf{x}=\mathbf{1}_{N}\otimes x$ for some $x \in \reals^{n}$. Then, on  $\mathbf{x}=\mathbf{r}$, the dynamics \eqref{eqn:distProposedAlg} reduces to $0=\dot{\mathbf{x}}-\dot{\mathbf{r}}= -\mathcal{R}^{T}\mathbf{F}(\mathbf{1}_{N}\otimes x)=-\mathcal{R}^{T}F(x)$, which implies $F(x)=0$, hence the largest invariant set is  the NE set. Since $V$ is radially unbounded, the conclusion follows.
\end{proof}

On the other hand, Assumption \ref{asmp:extendMono} can be quite restrictive. Instead of this assumption on $\mathbf{F}$, we will use a weaker additional condition, this time on the pseudo-gradient $F$. This is the inverse Lipschitz property. In the next section we discuss this property.

\section{Inverse Lipschitz} \label{sec:invLip}

In convex analysis and monotone operator theory there are three properties on an operator $T$ that are frequently used and are important. These three properties are: $\mu$-strongly monotone, $L$-Lipschitz, and $C$-cocoercive, which describe upper and lower bounds on an operator $T$. However, it appears there is a natural definition missing.

\begin{defn}
	An operator $T : \mathcal{H} \to 2^{\mathcal{H}}$ is $R$-inverse Lipschitz if,
\begin{align*}
	\norm{x-y} &\leq R \norm{Tx-Ty} \quad \forall x,y \in \mathcal{H},
\end{align*}
related to the condition used by Rockafeller \cite{rockInvLip}.
\end{defn}

\begin{remark}
	A $C$-cocoercive operator is also called $C$-inverse strongly monotone, because if $x=T^{-1}u$ and $y=T^{-1}v$ then,
	\begin{align*}
		\inp*{Tx - Ty}{x-y} &\geq C \snorm{Tx-Ty} \\
		\inp*{u - v}{T^{-1}u-T^{-1}v} &\geq C \snorm{u-v}
	\end{align*}
	This is the same as the inverse operator $T^{-1}$ being $C$-strongly monotone. In the same spirit, we call $T$ a $R$-inverse Lipschitz operator because it is the same as the inverse operator $T^{-1}$ being $R$-Lipschitz, i.e.,
	\begin{align*}
		\norm{x-y} &\leq R \norm{Tx-Ty} \\
		\norm{T^{-1}u-T^{-1}v} &\leq R \norm{u-v}
	\end{align*}
\end{remark}

\subsection{Similarities}
\subsubsection{Similarities in monotone operator theory}
The property of inverse Lipschitz is closely related to coercive or radially unbounded property,
\begin{defn}
	A function $f:\mathcal{H}\to\reals$ is coercive (radially unbounded) if,
\begin{align*}
	\lim_{\norm{x}\to \infty} f(x) = +\infty
\end{align*}
\end{defn}

Since coercive functions are real valued and $T$ is in general  vector valued, taking the norm can be thought of as an extension of the definition, i.e., $\lim_{\norm{x}\to \infty} \norm{Tx} = +\infty$. If $y=0$ and define $\tilde{T}x = Tx - T0$ then from the $R$-inverse Lipschitz definition of $T$, we see that $\norm{x} \leq R\norm{\tilde{T}x}$. Therefore, $R$-inverse Lipschitz is a stronger growth condition relating the input to the output, similar to coercivity and implies that $\tilde{T}$ is coercive.

\subsubsection{Similarities to optimization}

In optimization there are weaker conditions than strong convexity that can get linear convergence rates \cite{opt_weak_conditions}. One of these conditions is the Polyak-Lojasiewicz (PL) inequality. A function $f$ satisfies the (PL) inequality if $\forall x\in X$, $\frac{1}{2}\snorm{\nabla f(x)} \geq \mu \bracket{f(x)-f^{*}}$, where $f^*$ is the value of $f$ at the optimal solution. Theorem 2 \cite{opt_weak_conditions} shows that if $f$ has a Lipschitz-continuous gradient then (PL) is equivalent to the Error Bound (EB) inequality, $\forall x\in X$, $\norm{\nabla f (x)} \geq \mu \norm{x - Proj_{X^{*}}(x)}$. If $X = \reals^{n}$ then $\nabla f(Proj_{X^{*}}(x)) = 0$ and the condition can be written as, $\norm{\nabla f (x) - \nabla f (y)} \geq \mu \norm{x - y}$ where $y = Proj_{X^{*}}(x)$, hence $\nabla f$ is $\frac{1}{\mu}$-inverse Lipschitz. Thus, the $R$-inverse Lipschitz condition is the monotone operator equivalent to the Error Bound inequality / Polyak-Lojasiewicz inequality. If we replace $\nabla f $ with a monotone operator then this condition is the condition used by Rockafeller in \cite{rockInvLip}, 
\begin{align*}
	\norm{x-y} \leq R \norm{Tx-y}
\end{align*}
$\forall x\in \mathcal{H}$ and $y = Ty$ (restricted). If we remove the restriction of $y = Ty$ then we get the definition of $R$-inverse Lipschitz.

\subsubsection{Similarities in control / passivity} Passive systems with an inverse Lipschitz property have been analyzed in Chapter 6, section 11 \cite{control_inv_lip}. However, the analysis is only for the case when the operator is strongly monotone.

\subsection{Relations / Properties}

The following diagram shows the relationship between $R$-inverse Lipschitz and the other properties.

\begin{figure}[h]
\centering
\begin{tikzcd}[row sep=scriptsize, column sep=16ex]
& \mu \arrow[dd, "Prop.\ \ref{ul->c}"', bend right=15] \arrow[dr] \arrow[r, "Prop.\ \ref{u->r}"] & R \\
cvx \arrow[ur, "Prop.\ \ref{r->u}"] \arrow[dr, "Prop.\ \ref{l->c}"'] & & M \\
& C \arrow[uu, "Prop.\ \ref{cr->u}"', bend right=15] \arrow[ur] \arrow[r, "Prop.\ \ref{c->l}"'] & L
\end{tikzcd}
\caption{$cvx$: set of convex functions, $\mu$: set of $\mu$-strongly monotone operators, $C$: set of $C$-cocoercive operators, $M$: set of monotone operators, $R$: set of $R$-inverse Lipschitz operators, $L$: set of Lipschitz operators}
\end{figure}
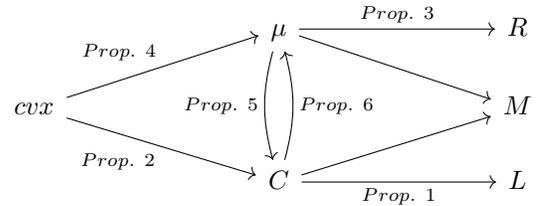

\begin{prop}\label{c->l}
	If an operator $T : \mathcal{H} \to 2^{\mathcal{H}}$ is $C$-inverse strongly monotone then it is $\frac{1}{C}$-Lipschitz
\end{prop}
\begin{proof}
Note that a $C$-inverse strongly monotone satisfies, $C \snorm{Tx-Ty} \leq \inp*{Tx-Ty}{x-y} \leq \norm{Tx-Ty}\norm{x-y}$, therefore  $\norm{Tx-Ty} \leq \frac{1}{C}\norm{x-y}$.
\end{proof}
\begin{prop}[Baillon-Haddad \cite{Baillon}] \label{l->c}
	If an operator $T : \mathcal{H} \to 2^{\mathcal{H}}$ is $\frac{1}{C}$-Lipschitz and is the gradient of a convex function then it is $C$-inverse strongly monotone.
\end{prop}
\begin{prop}\label{u->r}
	If an operator $T : \mathcal{H} \to 2^{\mathcal{H}}$ is $\mu$-strongly monotone then it is $\frac{1}{\mu}$-inverse Lipschitz
\end{prop}
\begin{proof}
Note that a $\mu$-inverse strongly monotone satisfies, $\mu \snorm{x-y} \leq \inp*{Tx-Ty}{x-y} \leq \norm{Tx-Ty}\norm{x-y}$, therefore  $\norm{x-y} \leq \frac{1}{\mu}\norm{Tx-Ty}$.
\end{proof}
\begin{prop}\label{r->u}
	If an operator $T : \mathcal{H} \to 2^{\mathcal{H}}$ is $\frac{1}{\mu}$-inverse Lipschitz and is the gradient of a convex function then it is $\mu$-strongly monotone.
\end{prop}
\begin{proof}
	Let $\partial f = T$. If $\partial f$ is $\frac{1}{\mu}$-inverse Lipschitz then $(\partial f)^{-1} = \partial f^{*}$ is $\frac{1}{\mu}$-Lipschitz. From Prop 12.60(a,b) \cite{RockafellarVarAnal}, if a function $f$ is convex and $\partial f$ is $\frac{1}{\mu}$-Lipschitz then $f^{*}$ is $\mu$-strongly monotone. Since $\partial f^{*}$ is $\frac{1}{\mu}$-Lipschitz and $f^{**} = f$ we can conclude that $f$ is $\mu$-strongly monotone.
\end{proof}

\begin{prop}\label{ul->c}
	If an operator $T : \mathcal{H} \to 2^{\mathcal{H}}$ is $\mu$-strongly monotone and $L$-Lipschitz then it is $\frac{\mu}{L^{2}}$-inverse strongly monotone, or cocoercive with $C = \frac{\mu}{L^{2}}$.
\end{prop}
\begin{proof}
	A $\mu$-strongly monotone operator satisfies, $\mu \snorm{x-y} \leq \inp*{Tx-Tx}{x-y}$ and a $L$-Lipschitz operator satisfies, $\frac{1}{L^{2}}\snorm{Tx-Ty}\leq \snorm{x-y}$. Combining these together gives $\frac{\mu}{L^{2}}\snorm{Tx-Ty} \leq \inp*{Tx-Tx}{x-y}$.
\end{proof}
\begin{prop}\label{cr->u}
	If an operator $T : \mathcal{H} \to 2^{\mathcal{H}}$ is $C$-inverse strongly monotone ($C$-cocoercive) and $R$-inverse Lipschitz then it is $\frac{C}{R^{2}}$-strongly monotone.
\end{prop}
\begin{proof}
	A $C$-inverse strongly monotone operator satisfies, $C\snorm{Tx-Ty} \leq \inp*{Tx-Tx}{x-y}$ and a $R$-inverse Lipschitz operator satisfies, $\frac{1}{R^{2}}\snorm{x-y}\leq \snorm{Tx-Ty}$. Combining these together gives $\frac{C}{R^{2}}\snorm{x-y} \leq \inp*{Tx-Tx}{x-y}$.
\end{proof}

The following Lemma is a useful property of operators under the $R$-inverse Lipschitz assumption.

\begin{prop} \label{lemma:resInvLip}
	Let $T : \mathcal{H} \to 2^{\mathcal{H}}$ be a maximally $\mu$-hypomonotone operator that is $R$-inverse Lipschitz. Then for any $\lambda \geq 0$ such that $\mu R^{2} \leq \lambda < \frac{1}{\mu}$, the following hold for the resolvent of $T$, $\mathcal{J}_{\lambda T} = (\Id + \lambda T)^{-1}$

\begin{enumerate}
	\item[(i)] $\mathcal{J}_{\lambda T}$ is maximally monotone.
	\item[(ii)] $\mathcal{J}_{\lambda T}$ is $L_{\mathcal{J}}$-Lipschitz, $\norm{\mathcal{J}_{\lambda T}x - \mathcal{J}_{\lambda T}y} \leq L_{\mathcal{J}}\norm{x-y}$ where, $L_{\mathcal{J}} \defeq \sqrt{\frac{R^{2}}{R^{2} + \lambda^{2} - 2\lambda \mu R^{2}}}$.
	\item[(iii)] $\inp*{x-y}{\mathcal{J}_{\lambda T}x - \mathcal{J}_{\lambda T}y} \leq \kappa_{\mathcal{J}} \snorm{x-y} $
	where,
	\begin{align*}
		\kappa_{\mathcal{J}} \defeq \begin{cases}
		\frac{R^{2}(1- \mu\lambda)}{R^{2} + \lambda^{2} - 2\mu\lambda R^{2}} & R\geq \lambda \\
		\frac{R^{2}(1+\frac{\lambda}{R})}{R^{2} + \lambda^{2} + 2\lambda R} & \lambda \geq R
		\end{cases}
	\end{align*} 
\end{enumerate}
\end{prop}

\begin{proof}
	Found in the Appendix.
\end{proof}

\begin{remark}
	The Lipschitz constant from Lemma \ref{lemma:resInvLip} (ii) can upper bound the inner product $\inp*{x-y}{\mathcal{J}_{\lambda T}x - \mathcal{J}_{\lambda T}y}$, but Lemma \ref{lemma:resInvLip} (iii) provides a tighter bound.
\end{remark}

When $T$ is differentiable some sufficient conditions for these properties are given next.

\begin{prop} \label{lemma:funcType}
	Let $T$ be a differentiable operator and the Jacobian of $T$ be denoted $JT(x)$. Then $T$ is,
	\begin{enumerate}
		\item $\mu$-strongly monotone if: \null\hfill $\frac{1}{2}\bracket{JT(x) + JT(x)^{T}} \succeq \mu I$
		\item $C$-cocoercive if: \null\hfill $C JT(x)^{T}JT(x) \preceq JT$
		\item $L$-Lipschitz if: \null\hfill $JT(x)^{T}JT(x) \preceq L^{2} I$
		\item $R$-inverse Lipschitz if: \null\hfill $JT(x)^{T}JT(x) \succeq \frac{1}{R^{2}} I$
\end{enumerate}
\end{prop}

\begin{proof}
	Found in the Appendix
\end{proof}

\subsection{Examples}

\begin{exmp}
	The operator $Tx = \begin{bmatrix}
	0 & -1 \\
	1 & 0
	\end{bmatrix}x$ is monotone but is not strongly monotone nor cocoercive. It is $1$-Lipschitz and $1$-inverse Lipschitz.
\end{exmp}
\begin{exmp}
	The operator $T : [-1,1] \to [-1,1]$, $Tx = x^{3}$ is $\frac{1}{3}$-cocoercive and is not strongly monotone nor inverse Lipschitz.
\end{exmp}
\begin{exmp}
	The operator $Tx = \begin{bmatrix}
	2 & 1 \\
	-1 & 3
	\end{bmatrix}x$ is $2$-strongly monotone, $\sqrt{\frac{15 + \sqrt{29}}{2}}$-Lipschitz, $\frac{1}{15+\sqrt{29}}$-cocoercive and $\sqrt{\frac{15 - \sqrt{29}}{2}}$-inverse Lipschitz
\end{exmp}
\begin{exmp}
	The operator $T: [0,\infty) \to [1,\infty)$, $Tx = e^{x}$ is $1$-strongly monotone and $1$-inverse Lipschitz, but is not cocoercive nor Lipschitz.
\end{exmp}
\begin{exmp}
	The operator $Tx = sin(x)$ is $1$-Lipschitz and is not strongly monotone, cocoercive or inverse Lipschitz.
\end{exmp}
\begin{exmp}
	The operator $T: (0,1) \to (0,\infty)$, $Tx = \frac{1}{x}$ is $1$-inverse Lipschitz and is not strongly monotone, cocoercive or Lipschitz.
\end{exmp}

\section{Convergence under Partial Information} \label{sec:partialConv}

We will now show that \eqref{eqn:distProposedAlg} converges to the NE when the monotonicity of the extended pseudo-gradient, Assumption \ref{asmp:extendMono}, is replaced by a weaker assumption only on the pseudo-gradient.
\begin{asmp}\label{asmp:inverse}
	The pseudo-gradient $F$ is $L_F$-Lipschitz, $R$-inverse Lipschitz, and $\mu$-hypomonotone, i.e., $\inp*{Fx-Fy}{x-y}\geq -\mu \snorm{x-y}$.
\end{asmp}
\begin{remark}
	Note that $F$ may not be monotone. For example, $F(x) = \begin{bmatrix}
	-1 & 1 \\ -1 & -1 \end{bmatrix}$ is $1$-hypomonotone, $\sqrt{2}$-Lipschitz and $\sqrt{2}$-inverse Lipschitz.
\end{remark}

When the extended monotonicity property (Assumption  \ref{asmp:extendMono}) does not hold, we use Assumption \ref{asmp:inverse}
and  take advantage of properties of the dynamics on the augmented consensus subspace and its orthogonal complement.  Our idea is to use a change of coordinates and in these coordinates show that, under Assumption \ref{asmp:monotone} and \ref{asmp:inverse}, the dynamics restricted to the consensus subspace satisfies a property similar to strict EIP for $\alpha$ parameters selected in a certain range (Lemma \ref{lemma:consTTS}). Then, for the overall dynamics,  we exploit this property together with the excess passivity of the Laplacian to balance the coupling terms  off the consensus subspace and show that \eqref{eqn:distProposedAlg} converges to a Nash Equilibrium (Theorem \ref{thm:distMono}).

We  first decompose the system into consensus and orthogonal component dynamics. Let $\mathbf{x}$ and $\mathbf{r}$ be decomposed into consensus and orthogonal components. i.e.,
\begin{align*}
	\mathbf{x} &= \mathbf{x}^{\orthc} + \mathbf{x}^{\orth}, & \mathbf{x}^{\orthc} &= \Pi_{\orthc}\mathbf{x}, & \mathbf{x}^{\orth} &= \mathbf{x} - \mathbf{x}^{\orthc}\\
	\mathbf{r} &= \mathbf{r}^{\orthc} + \mathbf{r}^{\orth}, & \mathbf{r}^{\orthc} &= \Pi_{\orthc}\mathbf{r}, & \mathbf{r}^{\orth} &= \mathbf{r} - \mathbf{r}^{\orthc}
\end{align*}
where $\Pi_{\orthc} = \frac{1}{N}(\mathbf{1}_{N}\otimes\mathbf{1}_{N}^{T}\otimes I_{n})$ and $\Pi_{\orth} = I_{Nn} - \Pi_{\orthc}$, $\mathbf{x}^{\orthc}=\mathbf{1}_{N}\otimes x$, $\mathbf{r}^{\orthc}=\mathbf{1}_{N}\otimes r$, for some $x,r\in \reals^n$. The overall dynamics \eqref{eqn:distProposedAlg} can be decomposed into the (augmented) consensus component dynamics,
\begin{align}
\begin{split}
	\mathbf{\dot{r}^{\orthc}} &= \alpha\bracket{\mathbf{x}^{\orthc} - \mathbf{r}^{\orthc}} \\
	\mathbf{\dot{x}^{\orthc}} &= -\frac{1}{N}\mathbf{1}_{N}\otimes\mathbf{F}\bracket{\mathbf{x}^{\orthc} + \mathbf{x}^{\orth}} - \beta\bracket{\mathbf{x}^{\orthc} - \mathbf{r}^{\orthc}}
\end{split} \label{eqn:consInter}
\end{align}
and the orthogonal component dynamics,
\begin{align} \label{eqn:orthDyn}
\begin{split}
	\mathbf{\dot{r}^{\orth}} &= \alpha\bracket{\mathbf{x}^{\orth} - \mathbf{r}^{\orth}} \\
	\mathbf{\dot{x}^{\orth}} &= -\Pi_{\orth}\mathcal{R}^{T}\mathbf{F}\bracket{\mathbf{x}^{\orthc} + \mathbf{x}^{\orth}} - \beta\bracket{\mathbf{x}^{\orth} - \mathbf{r}^{\orth}} \\
	&\qquad - c\mathbf{L}\mathbf{x}^{\orth}
\end{split}
\end{align}
which are coupled one to another via $\mathbf{x}^{\orth}$ and $\mathbf{x}^{\orthc}$. 

Let the change of variables $\mathbf{z}^{\orthc} := \mathbf{x}^{\orthc} - \mathbf{h}(\mathbf{r}^{\orthc})$ where 
\begin{align}\label{h_def}
\mathbf{h}(\mathbf{r}^{\orthc}) := \mathcal{J}_{\frac{1}{\beta N}\mathbf{1}_{N}\otimes\mathbf{F}}(\mathbf{r}^{\orthc}) = \left (\text{Id} + \frac{1}{\beta N}\mathbf{1}_{N}\otimes\mathbf{F} \right )^{-1}(\mathbf{r}^{\orthc})
\end{align}
is the resolvent of $\mathbf{1}_{N}\otimes\mathbf{F}$ on the consensus subspace. Then from \eqref{eqn:consInter}, it follows that 
\begin{align} \label{eqn:consDyn}
\mathbf{\dot{r}^{\orthc}} &= \alpha\bracket{\mathbf{z}^{\orthc} + \mathbf{h}(\mathbf{r}^{\orthc}) - \mathbf{r}^{\orthc}} \\
\mathbf{\dot{z}^{\orthc}} &= -\frac{1}{N}\mathbf{1}_{N}\otimes\mathbf{F}\bracket{\mathbf{z}^{\orthc} + \mathbf{h}(\mathbf{r}^{\orthc})+ \mathbf{x}^{\orth}} \notag \\
&\qquad - \bracket{\beta + \alpha \frac{\partial \mathbf{h}}{\partial \mathbf{r}}} \bracket{\mathbf{z}^{\orthc} + \mathbf{h}(\mathbf{r}^{\orthc}) - \mathbf{r}^{\orthc}}  \notag
\end{align}
Therefore, the dynamics \eqref{eqn:distProposedAlg} can be equivalently represented as \eqref{eqn:orthDyn} and \eqref{eqn:consDyn}. Note that an equilibrium point for these dynamics is $(\mathbf{\bar{z}^{\orthc}},\mathbf{\bar{r}^{\orthc}},\mathbf{\bar{x}^{\orth}},\mathbf{\bar{r}^{\orth}}) = (\mathbf{0}_{Nn}, \mathbf{1}_{N}\otimes x^{*}, \mathbf{0}_{Nn}, \mathbf{0}_{Nn})$,  where  $F(x^*)=0$ ($x^*$ is a NE), cf. Lemma \ref{lemma:distEQ}.

Consider the dynamics \eqref{eqn:consDyn} restricted to the consensus subspace, i.e., when $\mathbf{x}^{\orth} = \mathbf{0}$, which is given as 
\begin{align} \label{eqn:consDynFI}
\mathbf{\dot{r}^{\orthc}} &= \alpha \bracket{\mathbf{z}^{\orthc} + \mathbf{h}(\mathbf{r}^{\orthc}) - \mathbf{r}^{\orthc}}\\
	\mathbf{\dot{z}^{\orthc}} &= -\frac{1}{N}\mathbf{1}_{N}\otimes\mathbf{F}\bracket{\mathbf{z}^{\orthc} + \mathbf{h}(\mathbf{r}^{\orthc})} \notag \\ 
	&\qquad - \bracket{\beta + \alpha \frac{\partial \mathbf{h}}{\partial \mathbf{r}}} \bracket{\mathbf{z}^{\orthc} + \mathbf{h}(\mathbf{r}^{\orthc}) - \mathbf{r}^{\orthc}}  \notag
\end{align}

\begin{lemma} \label{lemma:consTTS}
Consider \eqref{eqn:consDynFI}, under Assumption \ref{asmp:Jsmooth} and \ref{asmp:inverse}. For any $0<d < 1$,  let $\beta \in \bracket{\frac{\mu}{N}, \frac{1}{\mu N R^{2}}}$ and
\begin{align*}
	0< \alpha < \frac{4 d(1-d)(\beta-\frac{\mu}{N})(1-{\kappa_{\mathcal{J}}})}{\left((1-d)+ d(L_{\mathcal{J}} + L^2_{\mathcal{J}})\right )^2}
\end{align*}
where $\kappa_{\mathcal{J}}$ and $L_{\mathcal{J}}$ are obtained from Lemma \ref{lemma:resInvLip} for the pseudo-gradient $F$ and $\lambda = \frac{1}{\beta N}$. Let,
\begin{align*}
V^{\orthc}(\mathbf{z}^{\orthc}, \mathbf{r}^{\orthc}) = \frac{1-d}{2}\snorm{\mathbf{r}^{\orthc} - \mathbf{\bar{r}}^{\orthc}} + \frac{d}{2}\snorm{\mathbf{z}^{\orthc}}
\end{align*}
where $\mathbf{\bar{r}}^{\orthc} = \mathbf{1}\otimes x^{*}$. Then, along any solution of \eqref{eqn:consDynFI}, $\dot{V}^{\orthc} \leq -\varpi^{T} \Phi\varpi$ where $\varpi = \bracket{\norm{\mathbf{r}^{\orthc} - \mathbf{\bar{r}}^{\orthc}}, \norm{\mathbf{z}^{\orthc}}}$,
\begin{align} \label{eqn:lyapNSD}
	\Phi &= \begin{bmatrix}
	(1-d)\alpha(1-\kappa_{\mathcal{J}}) & -\frac{\alpha +\alpha(L_{\mathcal{J}}^{2}+L_{\mathcal{J}}-1)d}{2} \\
	-\frac{\alpha + \alpha(L_{\mathcal{J}}^{2}+L_{\mathcal{J}}-1)d}{2} & d\bracket{\beta - \frac{\mu}{N}}
	\end{bmatrix}
\end{align}
and the matrix $\Phi$ is positive definite. 
\end{lemma}

Using this Lemma we can show that \eqref{eqn:proposedAlg}, in the full information case, converges for hypomonotone games instead of just monotone.

\begin{lemma} \label{lemma:full_hypo}
Consider \eqref{eqn:proposedAlg}, under Assumption \ref{asmp:Jsmooth} and \ref{asmp:inverse}. For any $0<d < 1$,  let $\beta \in \bracket{\mu, \frac{1}{\mu R^{2}}}$ and
\begin{align*}
	0< \alpha < \frac{4 d(1-d)(\beta-\frac{\mu}{N})(1-{\kappa_{\mathcal{J}}})}{\left((1-d)+ d(L_{\mathcal{J}} + L^2_{\mathcal{J}})\right )^2}
\end{align*}
where $\kappa_{\mathcal{J}}$ and $L_{\mathcal{J}}$ are obtained from Lemma \ref{lemma:resInvLip} for the pseudo-gradient $F$ and $\lambda = \frac{1}{\beta }$. Then, the dynamics \eqref{eqn:proposedAlg} globally converge to a NE $x^{*}$. 
\end{lemma}

Next, we now show that \eqref{eqn:distProposedAlg} converges to a NE in the partial information case.

\begin{thm} \label{thm:distMono}
Consider a game $\mathcal{G}(\mathcal{N},J_i,\Omega_i)$ over a communication graph $G_c = (\mathcal{N},\mathcal{E})$, under Assumption \ref{asmp:Jsmooth}, \ref{asmp:graph} and \ref{asmp:inverse}. Let the overall dynamics of the agents be given by \eqref{eqn:distProposedAlg} or, equivalently, \eqref{eqn:orthDyn} and \eqref{eqn:consDyn}. Given any  $0<d <  1$, set $\alpha, \beta$ to satisfy the conditions in Lemma \ref{lemma:consTTS}. Set $c$ such that, 
\begin{align} \label{eqn:graphGameCondition}
	c\lambda_{2}(L) > \frac{\eta_{1} + \eta_{2}}{4 det(\Phi)}L_F^{2} + L_F
\end{align} 
where $\Phi$ is defined in \eqref{eqn:lyapNSD} and
\begin{align*}
\eta_{1} &= \alpha(1-d)(1-{\kappa_{\mathcal{J}}})\bracket{1+\frac{d}{\sqrt{N}}}^{2} + d\bracket{\beta - \frac{\mu}{N}}L_{\mathcal{J}}^{2}  \\
\eta_{2} &= \alpha \bracket{1 + [L_{\mathcal{J}}^{2} + L_{\mathcal{J}} - 1]d}\bracket{1+\frac{d}{\sqrt{N}}}L_{\mathcal{J}}
\end{align*}
Then, the dynamics \eqref{eqn:distProposedAlg} globally converges to a NE $x^*$.
\end{thm}
\begin{proof}
	Consider the candidate Lyapunov function,
\begin{align*}
V(\mathbf{z}^{\orthc},\mathbf{r}^{\orthc},\mathbf{x}^{\orth},\mathbf{r}^{\orth}) &= \frac{1-d}{2}\snorm{\mathbf{r}^{\orthc} - \mathbf{\bar{r}}^{\orthc}} + \frac{d}{2}\snorm{\mathbf{z}^{\orthc}} \\
	&\quad + \frac{1}{2}\snorm{\mathbf{x}^{\orth}} + \frac{\beta}{2\alpha}\snorm{\mathbf{r}^{\orth}}
\end{align*}
where $\mathbf{\bar{r}}^{\orthc} = \mathbf{1}_{N}\otimes x^{*}$ and $\mathbf{F}(\mathbf{\bar{r}}^{\orthc}) =F(x^*)= 0$. 
Along \eqref{eqn:orthDyn} and \eqref{eqn:consDyn}, after re-grouping terms we can write,
\begin{align*}
	\dot{V} &= \alpha(1-d)\inp*{\mathbf{r}^{\orthc} - \mathbf{\bar{r}}^{\orthc}}{\mathbf{z}^{\orthc} + \mathbf{h}(\mathbf{r}^{\orthc}) - \mathbf{r}^{\orthc}} \\
	& - \frac{d}{N}\inp*{\mathbf{z}^{\orthc}}{\mathbf{1}_{N}\otimes\mathbf{F}(\mathbf{z}^{\orthc} + \mathbf{h}(\mathbf{r}^{\orthc}))} \\
	&- d\inp*{\mathbf{z}^{\orthc}}{\beta(\mathbf{z}^{\orthc} + \mathbf{h}(\mathbf{r}^{\orthc}) - \mathbf{r}^{\orthc}) + \frac{\partial \mathbf{h}}{\partial \mathbf{r}}\mathbf{\dot{r}^{\orthc}}} \\
	&-  \frac{d}{N}\inp*{\mathbf{z}^{\orthc}}{\mathbf{1}_{N}\otimes\mathbf{F}(\mathbf{z}^{\orthc} + \mathbf{h}(\mathbf{r}^{\orthc}) + \mathbf{x}^{\orth})} \\
	&+  \frac{d}{N}\inp*{\mathbf{z}^{\orthc}}{\mathbf{1}_{N}\otimes\mathbf{F}(\mathbf{z}^{\orthc} + \mathbf{h}(\mathbf{r}^{\orthc}))} - \beta\snorm{\mathbf{x}^{\orth}-\mathbf{r}^{\orth}} \\
	&- \inp*{\mathbf{x}^{\orth}}{\Pi_{\orth}\mathcal{R}^{T}\mathbf{F}(\mathbf{z}^{\orthc} + \mathbf{h}(\mathbf{r}^{\orthc}) + \mathbf{x}^{\orth}) + c\mathbf{L}\mathbf{x}^{\orth}}
\end{align*}
Note that the first three terms correspond to $\dot{V}^{\orthc}$ along \eqref{eqn:consDynFI} in Lemma \ref{lemma:consTTS}, and $\beta>0$. Therefore, using Lemma \ref{lemma:consTTS} yields,
\begin{align*}
	\dot{V} &\leq -\omega^{T} 
\Phi 	
\omega
 \\
 	&-  \frac{d}{N}\inp*{\mathbf{z}^{\orthc}}{\mathbf{1}_{N}\otimes\mathbf{F}(\mathbf{z}^{\orthc} + \mathbf{h}(\mathbf{r}^{\orthc}) + \mathbf{x}^{\orth})  - \mathbf{1}_{N}\otimes\mathbf{F}(\mathbf{z}^{\orthc} + \mathbf{h}(\mathbf{r}^{\orthc}))} \\
	&- \inp*{\mathbf{x}^{\orth}}{\Pi_{\orth}\mathcal{R}^{T}\mathbf{F}(\mathbf{z}^{\orthc} + \mathbf{h}(\mathbf{r}^{\orthc}) + \mathbf{x}^{\orth})} - c\lambda_{2}(L)\snorm{\mathbf{x}^{\orth}} 
\end{align*}
where $\omega  = \bracket{\norm{\mathbf{r}^{\orthc} - \mathbf{\bar{r}}^{\orthc}}, \norm{\mathbf{z}^{\orthc}}}$. Under Assumption \ref{asmp:inverse}, it follows that  $\mathbf{F}$ is also $L_F$-Lipschitz, (cf. Lemma 3, \cite{gramProx} or Lemma 1,\cite{TatarenkoShi}). Using this and Cauchy-Schwarz inequality, as well as $\Pi_{\orth}\mathcal{R}^{T}\mathbf{F}(\mathbf{h}(\mathbf{\bar{r}}^{\orthc})) = \mathbf{0}_{Nn}$  yields, 
\begin{align*}
	\dot{V} 
	&\leq -\omega^{T}
	\Phi
	\omega  + \frac{d}{N} \sqrt{N} L_F \norm{\mathbf{z}^{\orthc}} \norm{\mathbf{x}^{\orth}} - c\lambda_{2}(L)\snorm{\mathbf{x}^{\orth}} \\
	&+ \norm{\mathbf{x}^{\orth}}\norm{\Pi_{\orth}\mathcal{R}^{T}}\norm{\mathbf{F}(\mathbf{z}^{\orthc} + \mathbf{h}(\mathbf{r}^{\orthc}) + \mathbf{x}^{\orth}) -
	\mathbf{F}(\mathbf{h}(\mathbf{\bar{r}}^{\orthc}))} 
\end{align*}which, with $\norm{\Pi_{\orth}\mathcal{R}^{T}} \leq 1$ and Lemma \ref{lemma:resInvLip}(ii) for $\mathbf{h}$, leads to,
\begin{align*}
	\dot{V} 
	&\leq -\omega^{T}
	\Phi
	\omega  + \frac{d}{N} \sqrt{N} L_F \norm{\mathbf{z}^{\orthc}} \norm{\mathbf{x}^{\orth}} - c\lambda_{2}(L)\snorm{\mathbf{x}^{\orth}} \\
	&+ \norm{\mathbf{x}^{\orth}} L_F\bracket{\norm{\mathbf{z}^{\orthc}} + L_{\mathcal{J}}\norm{\mathbf{r}^{\orthc} - \mathbf{\bar{r}}^{\orthc}} + \norm{\mathbf{x}^{\orth}}}
\end{align*}

Therefore,
\begin{align*}
	&\dot{V} \leq 
	 -\hat{\omega}^{T} \begin{bmatrix}
	& \hspace{-1cm}  \Phi  
	& -\frac{L_FL_{\mathcal{J}}}{2} \\
	& & -\frac{L_F(\sqrt{N}+d)}{2\sqrt{N}} \\
	-\frac{L_FL_{\mathcal{J}}}{2} & -\frac{L_F(\sqrt{N}+d)}{2\sqrt{N}} & c\lambda_{2}(L)-L_F
	\end{bmatrix} \hat{\omega}
\end{align*}
where $\hat{\omega}:= \bracket{\omega, \norm{\mathbf{x}^{\orth}}}= \bracket{\norm{\mathbf{r}^{\orthc} - \mathbf{\bar{r}}^{\orthc}}, \norm{\mathbf{z}^{\orthc}}, \norm{\mathbf{x}^{\orth}}}$. 
The block matrix is positive definite if its Schur complement is positive definite, i.e., if 
\begin{align*}
	c\lambda_{2}(L) > \frac{\eta_{1} + \eta_{2}}{4 det(\Phi)}L_{F}^{2} + L_{F}
\end{align*}
where $\eta_{1}$, $\eta_{2}$ are as in the  statement.  Therefore, $\dot{V} \leq 0$ and $\dot{V} = 0$ only if ${\mathbf{r}^{\orthc} = \mathbf{\bar{r}}^{\orthc}}=\mathbf{1}_{N}\otimes x^*$,  $\mathbf{z}^{\orthc} =0$, $\mathbf{x}^{\orth}=0$, i.e.,  $\mathbf{x}^{\orthc} = 0+ \mathbf{h}(\mathbf{\bar{r}}^{\orthc})= \mathbf{h}(\mathbf{1}_{N}\otimes x^*)= \mathbf{1}_{N}\otimes h(x^*)=\mathbf{1}_{N}\otimes x^*$, where since $F(x^*)=0$, $x^*$ is a NE. The conclusion follows by a LaSalle argument \cite{nonlinear}. 
\end{proof}


The conditions that we obtain for Theorem \ref{thm:distMono} are conservative. In the following section we restrict our attention to an important subclass of games called quadratic games and derive tighter conditions on the parameters $\alpha, \beta$ to ensure convergence.

\section{Quadratic Hypomonotone Games} \label{sec:Quad}

In this section, we consider a quadratic game $J_{i}(x_{i},x_{-i})= \frac{1}{2} x^TQ_{i}x + l_{i}^{T}x + c_{i}$ where $Q_{i}= Q_{i}^{T}\in\reals^{n\times n}$, $l_{i}\in\reals^{n}$, and $c_{i} \in \reals$. The gradient of agents cost function with respect to their own action is, $\nabla_{x_{i}}J_{i}(x) = Q_{i}x + l_{i}$ and the pseudo-gradient is,
\begin{align}
	F(x) &= Ax + b ,\qquad
	A \defeq\begin{bmatrix}
		Q_{1} \\
		Q_{2} \\
		\vdots \\
		Q_{N}
	\end{bmatrix}, \quad
	b \defeq \begin{bmatrix}
		l_{1} \\
		l_{2} \\
		\vdots \\
		l_{N}
	\end{bmatrix} \label{eqn:quad_pseudo}
\end{align}

For the perfect information case, algorithm \eqref{eqn:proposedAlg}, after the change of coordinates, $\hat{x} = x-x^*$ and $\hat{r} = r - r^{*}$, is written as,
\begin{align} \label{eqn:DynLT}
	\dot{w} = \begin{bmatrix}
	\dot{\hat{x}} \\
	\dot{\hat{r}}
	\end{bmatrix} &= \begin{bmatrix}
		-A-\beta I & \beta I \\
		\alpha I & -\alpha I
	\end{bmatrix}\begin{bmatrix}
	\hat{x} \\
	\hat{r}
	\end{bmatrix} \defeq M w
\end{align}

The following lemma relates the eigenvalues of $A$ to the eigenvalues of the overall $M$, \eqref{eqn:DynLT}. 

\begin{lemma} \label{lemma:eigRelation}
	Let $A\in \reals^{n\times n}$ be a matrix where the $i^{th}$ eigenvalue of $A$ is denoted $\rho_{i}$. Then the eigenvalues of $M$, \eqref{eqn:DynLT}, are,
	\begin{align}
		\lambda_{i} &= \frac{-(\alpha + \beta + \rho_{i}) \pm \sqrt{(\alpha + \beta + \rho_{i})^{2} - 4\alpha \rho_{i}}}{2} \label{eqn:eigMap}
	\end{align}
for all $i \in \set{1,\dots,n}$.	
\end{lemma}
\begin{proof}
	Found in the Appendix
\end{proof}

The following Lemma gives conditions for the eigenvalues of $M$ to be in the OLHP.
\begin{lemma} \label{thm:stableA}
Let $A\in \reals^{n\times n}$ be a matrix where the $i^{th}$ eigenvalue of $A$ is denoted $\rho_{i} = r_{i} + \mathfrak{j}k_{i}$ where $r_{i}$ ($k_{i}$) is the real (imaginary) part of $\rho_{i}$ and $\mathfrak{j} = \sqrt{-1}$. 
\begin{enumerate}
	\item[(i)] If $\rho_{i} = 0$ and $\alpha, \beta > 0$, then $\lambda_{i}$ from \eqref{eqn:eigMap} are $0$ and $-(\alpha+\beta)$.
	\item[(ii)] If $\rho_{i} \neq 0$, $r_{i} \geq 0$ and $\alpha, \beta > 0$, then $\lambda_{i}$ from \eqref{eqn:eigMap} are complex conjugate with real part less than $0$.
	\item[(iii)] If $r_{i} < 0$, $\beta \in \left(-r_i,\frac{k_i^2 + r_i^2}{-r_i}\right)$ and
	\begin{align*}
	\alpha \in \left(0,-(\beta+r_{i}) + \sqrt{\frac{(\beta+r_i)k_i^2}{-r_i}}\right)
\end{align*}
then $\lambda_{i}$ from \eqref{eqn:eigMap} have real part less than $0$.
\end{enumerate}	
\end{lemma}
\begin{proof}
	Found in the Appendix
\end{proof}

\begin{remark}
	Note that if the eigenvalues of $A$ fall only in case (i) and (ii) then $F$ is monotone. Additionally, the conditions $\alpha, \beta \geq 0$ are the same conditions as for the nonlinear case, Theorem \ref{thm:EIPconvergence}. If $A$ has eigenvalues in case (iii) then $F$ is hypomonotone. If the eigenvalues of $A$ are $-r\pm \mathfrak{j}k$, then $F$ is $r$-hypomonotone and $R=\frac{1}{\sqrt{r^{2}+k^{2}}}$-inverse Lipschitz. From Lemma \ref{thm:stableA}, $\beta \in (\mu, \frac{1}{\mu R^{2}})$ is the same condition on $\beta$ as in Lemma \ref{lemma:full_hypo} for the nonlinear case.
\end{remark}

\begin{thm} \label{thm:fullQuad}
	Consider a quadratic game $\mathcal{G}(\mathcal{N},J_i,\Omega_i)$ under Assumption \ref{asmp:Jsmooth}. Let the overall dynamics of the agents  be given by \eqref{eqn:proposedAlg}. For the matrix $A$ given in \eqref{eqn:quad_pseudo} with eigenvalues $\rho_{i} = r_{i} + \mathfrak{j}k_{i}$, let $\mathcal{I}=\setc{i\in\set{1,\dots,n}}{r_{i}<0}$. If $\mathcal{I}=\emptyset$ then set $\alpha, \beta > 0$ else,
\begin{align} \label{eqn:alpha_beta_conditions}
\begin{split}
	\beta &\in \bigcap_{i\in \mathcal{I}}\left(-r_i,\frac{k_i^2 + r_i^2}{-r_i}\right) \\
	\alpha &\in \bigcap_{i\in \mathcal{I}}\left(0,-\bracket{\beta+r_i} + \sqrt{\frac{(\beta+r_i)k_i^2}{-r_i}}\right)
\end{split}
\end{align}
	Then, the set $\setc{(x^*,x^{*})}{F(x^*) = 0}$  is globally asymptotically stable.
\end{thm}

\begin{conj}
	For the class of quadratic games where $F$ is $R$-inverse Lipschitz (for the perfect information setting) the optimal convergence rate is $exp(\frac{-1}{3R}t)$ when $\alpha = \frac{5}{9R}$ and $\beta = \frac{4}{9R}$.
\end{conj}

\subsection{Partial Information}

In the partial information case the dynamics \eqref{eqn:distProposedAlg} are,
\begin{align} \label{eqn:distDynamics}
\begin{split}
	\mathbf{\dot{x}} &= -\mathcal{R}^{T}(\mathbf{A}\mathbf{x}+b) - \beta(\mathbf{x}-\mathbf{r}) - c\mathbf{L}\mathbf{x} \\
	\mathbf{\dot{r}} &= \alpha(\mathbf{x}-\mathbf{r})
\end{split}
\end{align}

Similar to the complete information case, after doing a change of coordinates, we can prove convergence of \eqref{eqn:distDynamics}.

\begin{thm} \label{thm:distQuad}
	Consider a game $\mathcal{G}(\mathcal{N},J_i,\Omega_i)$ under Assumption \ref{asmp:Jsmooth}, \ref{asmp:graph}, and \ref{asmp:inverse}. Let the overall dynamics of the agents be given by \eqref{eqn:distProposedAlg}. Let $\alpha, \beta$ be selected as in \eqref{eqn:alpha_beta_conditions} and scaled by $\frac{1}{N}$, and $c$ such that,
\begin{align} \label{eqn:GraphCondition}
		c\lambda_{2}(L) \geq L_{\mathbf{A}} + \bracket{L_{\mathbf{A}}\bracket{\frac{p}{\sqrt{N}} + \frac{1}{2}}}^{2}
	\end{align}
	where $L_{\mathbf{A}} = \norm{\mathbf{A}}$, $p = \norm{P}$ where $P \succ \mathbf{0}$ satisfies the Lyapunov equation $P\tilde{M} + \tilde{M}^{T}P = -I$ and
	\begin{align*}
		\tilde{M} &= \begin{bmatrix}
			-\frac{1}{N}A - \beta I & \beta I \\
			\alpha I & -\alpha I
		\end{bmatrix}
	\end{align*}
Then, the set $\setc{(\mathbf{1}\otimes x^*, \mathbf{1}\otimes x^{*})}{F(x^*) = 0}$  is globally asymptotically stable.
\end{thm}
\begin{proof}
	Found in the Appendix.
\end{proof}

\begin{remark}
Note that Theorem \ref{thm:distQuad} requires $\norm{P}= p$, if we restrict $\beta = \alpha$, and use Corollary 1 \cite{lyapBound} and Corollary 2.10 \cite{blockNorm}, we can obtain the simpler bound
\begin{align} \label{eqn:lyapBound}
	p \leq \frac{N}{2L_{A} + 4\alpha N}
\end{align}
\end{remark}

\subsection{Comparing Results For Quadratic vs General Games}

For perfect information quadratic games with monotone pseudo-gradient, notice that Theorem \ref{thm:fullQuad} requires that $\alpha, \beta > 0$ and the rate of convergence can be determined by Lemma \ref{lemma:eigRelation}. For perfect information general games with monotone pseudo-gradient, Theorem \ref{thm:EIPconvergence} also requires $\alpha, \beta > 0$ but with no rate of convergence.

For the partial information quadratic games with monotone pseudo-gradient, Theorem \ref{thm:distQuad} again requires that $\alpha, \beta > 0$. Additionally, the theorem requires that $c$ is larger than a function of the Lipschitz constant of the pseudo-gradient. For partial information general games with monotone pseudo-gradient, Theorem \ref{thm:distMono} allows $\beta > 0$ but $\alpha$ is now restricted by a function of $\beta$. Additionally, the $c$ term is larger than the one obtained for quadratic games.

For perfect information quadratic games with hypomonotone pseudo-gradient, Lemma \ref{thm:stableA} and Theorem \ref{thm:distQuad} provides tight conditions on the range of values of $\alpha$ and $\beta$ for convergence to a NE. Note that for quadratic games, we are able to use the same method of analyzing the eigenvalues for both monotone and hypomonotone games. On the other hand for general games, the EIP analysis cannot be extended to the hypomonotone case and a different method is used to prove convergence. The analysis ends up having restrictions on $\alpha$ that don't appear for the quadratic case. The quadratic game case suggests that there might be a better Lyapunov function that could remove or relax the condition on $\alpha$ for general games.

\section{Simulations} \label{sec:sim}
In this section we first consider three hypomonotone quadratic games between $N=10$ agents communicating over a ring $G_c$ graph. We index each game by $ \mathcal{G}_{1}, \mathcal{G}_{2}, \mathcal{G}_{3}$. In game $\mathcal{G}_{j}$, the cost function for agent $i$ is $J_{i}(x) = w_{i}^{\mathcal{G}_{j}} x_{i}^{T}\begin{bmatrix} 5 & 1 \\ -1 & 5 \end{bmatrix} x_{N+1-i}$ where $w^{\mathcal{G}_{j}} = [w_{1}^{\mathcal{G}_{j}},\dots, w_{N}^{\mathcal{G}_{j}}]$ is equal to
\begin{align*}
	w^{\mathcal{G}_{1}} &= \begin{bmatrix}
	1 & 1 & 1 & 1 & 1 & -1 & -1 & -1 & -1 & -1
	\end{bmatrix} \\
	w^{\mathcal{G}_{2}} &= \begin{bmatrix}
	\frac{-9}{9} & \frac{-7}{9} & \frac{-5}{9} & \frac{-3}{9} & \frac{-1}{9} &\frac{1}{9} & \frac{3}{9} & \frac{5}{9} & \frac{7}{9} & \frac{9}{9}
	\end{bmatrix} \\
	w^{\mathcal{G}_{3}} &= \begin{bmatrix}
	-2 & -1 & -1 & -1 & -1 & 1 & 1 & 1 & 1 & 2
	\end{bmatrix}
\end{align*}
For game $\mathcal{G}_{1}$ the eigenvalues of $A$ from \eqref{eqn:quad_pseudo} are $1\pm \mathfrak{j}5$; $\mathcal{G}_{2}$ the eigenvalues are $\pm 1\pm \mathfrak{j}5$, $\pm \frac{7}{9}\pm \mathfrak{j}\frac{35}{9}$, $\pm \frac{5}{9}\pm \mathfrak{j}\frac{25}{9}$, $\pm \frac{3}{9}\pm \mathfrak{j}\frac{15}{9}$, and $\pm \frac{1}{9}\pm \mathfrak{j}\frac{5}{9}$; and for $\mathcal{G}_{3}$ the eigenvalues are $\pm 2\pm \mathfrak{j}10$ and $\pm 1\pm \mathfrak{j}5$. For all three games the Nash equilibrium is the origin. The following table contains information about the parameter values as in Theorem \ref{thm:distMono} and \ref{thm:distQuad}. For game $\mathcal{G}_{2}$ the conditions of Lemma \ref{lemma:resInvLip} are not satisfied and hence the column is empty. The $\beta$ values are selected as $0.9 \beta_{min} + 0.1 \beta_{max}$ and the $\alpha$ values are selected as $0.5 \alpha_{min} + 0.5 \alpha_{max}$.

\begin{center}
\begin{tabular}{ |c|c c|c|c c| } \hline
 & \multicolumn{2}{|c|}{$\mathcal{G}_{1}$} & $\mathcal{G}_{2}$ & \multicolumn{2}{|c|}{$\mathcal{G}_{3}$} \\ \hline
 param. & Thm \ref{thm:distQuad} & Thm \ref{thm:distMono} & Thm \ref{thm:distQuad} & Thm \ref{thm:distQuad} & Thm \ref{thm:distMono} \\ \hline
 $\beta_{min}$ & 0.1 & 0.1 & 0.1 & 0.2 & 0.2\\ 
 $\beta_{max}$ & 2.6 & 2.6 & $\frac{13}{45}$ & 2.6 & 1.3 \\  
 $\beta$ & 0.35 & 0.35 &  $\frac{107}{900}$ & 0.44 & 0.31\\ \hline 
 $d$ &  & 0.5 & & & 0.5\\ \hline 
 $\alpha_{min}$ & 0 & 0 & 0 & 0 & 0\\
 $\alpha_{max}$ & 0.540 & 0.145 & 0.065 & 0.581 & 0.064\\  
 $\alpha$ & 0.270 & 0.072 & 0.032 & 0.290 & 0.032\\ \hline
 $c_{min}$ & 1517 & 1668 & $2.22\times 10^6$ & 7739 & 15057 \\ \hline
\end{tabular}
\end{center}

Figure \ref{fig:actions} shows the action trajectories for game $\mathcal{G}_{1}$ under \eqref{eqn:proposedAlg} for the parameters $\alpha, \beta, c$ satisfying Theorem \ref{thm:distQuad}, where the initial conditions $x(0)$, $r(0)$ are randomly selected with components between  $-10$ to $10$. Notice in $\mathcal{G}_{1}$ that $\beta$ used is the same for Theorem \ref{thm:distMono} and Theorem \ref{thm:distQuad}. However, the $\alpha$ obtained from Theorem \ref{thm:distMono} gives a conservative value for $\alpha$ and is an order of magnitude smaller than Theorem \ref{thm:distQuad}.

\begin{figure}[ht]
\centering
\begin{minipage}[ht]{1\columnwidth}
	\centerline{\includegraphics[width=7cm]{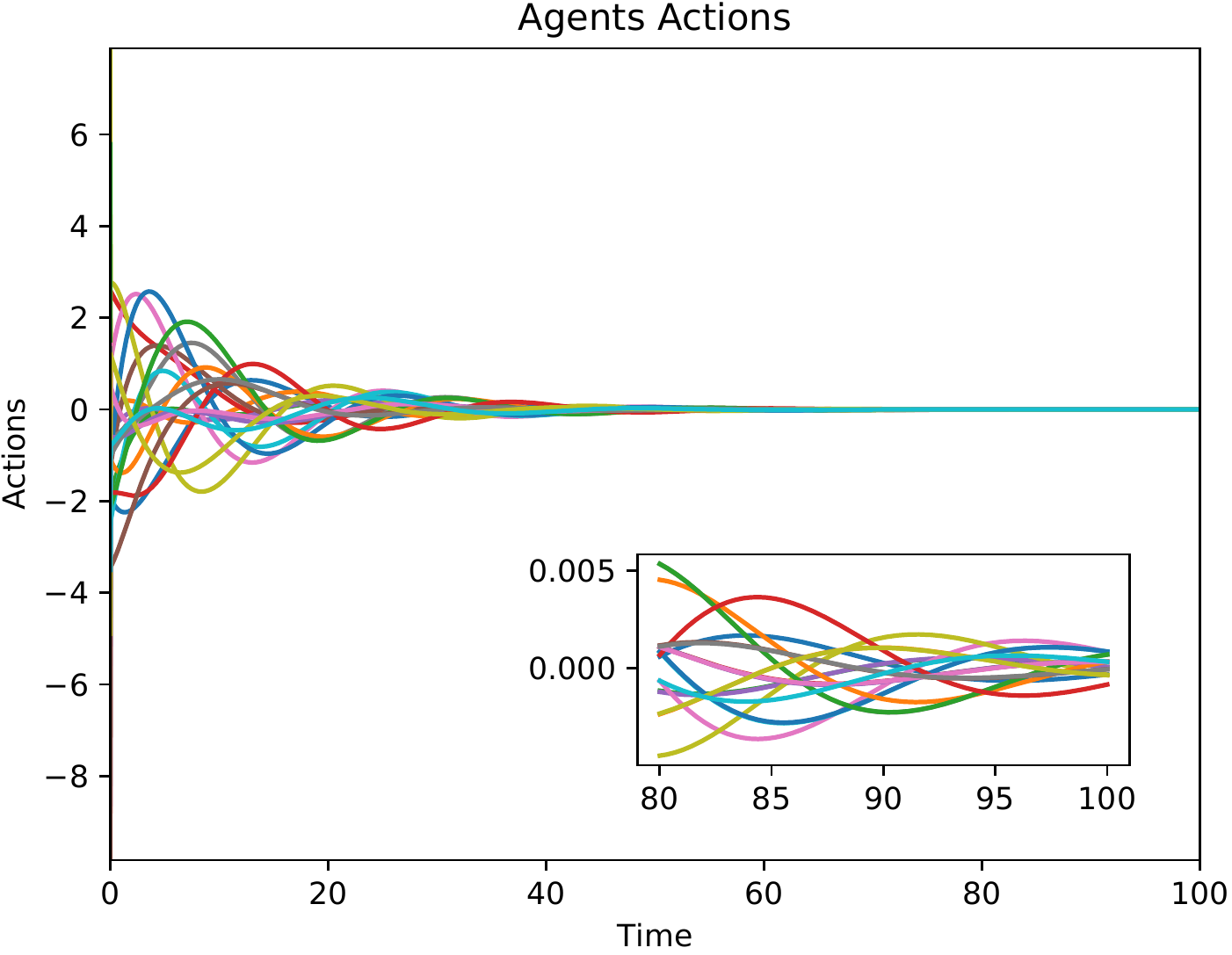}}
	\caption{Evolution of agents' actions }\label{fig:actions}
\end{minipage}
\end{figure}


The figures for the other examples are similar and are omitted.

\subsection{Nonquadratic Example}
The following example is a non quadratic game where Theorem \ref{thm:distQuad} no longer applies. Consider a hypomonotone game between $N=10$ agents communicating over a ring $G_c$ graph. The cost function for agent $i$ is $J_{i}(x) = w_{i}^{1} x_{i}^{T}\begin{bmatrix} 5 & 0 \\ 0 & 5 \end{bmatrix} x_{N+1-i} + w_{i}^{1}x_{i}^{T}\begin{bmatrix} \sin(x_{N+1-i,2}) \\ -\sin(x_{N+1-i,1})\end{bmatrix}$  where $x_{i,j}$ is the $j$th component of the vector $x_{i}$. For this game the pseudo-gradient is $1$-hypomonotone, $\frac{1}{4}$-inverse Lipschitz, and $6$-Lipschitz. 

Using Theorem \ref{thm:distMono}, $\beta_{min}=0.1$, $\beta_{max}=1.6$, and we selected $\beta = 0.9 \beta_{min} + 0.1 \beta_{max} = 0.25$. Using $d=0.5$ we obtain that $\alpha_{min} = 0$, $\alpha_{max} = 0.095$ and $\alpha = 0.5 \alpha_{min} + 0.5 \alpha_{max} = 0.0478$. Lastly, for a ring communication graph we obtain that $c_{min} = 3417$. Figures \ref{fig:actions_nonlin} shows the action trajectories and convergence to the NE. 

\begin{figure}[ht]
\centering
\begin{minipage}[ht]{1\columnwidth}
	\centerline{\includegraphics[width=7cm]{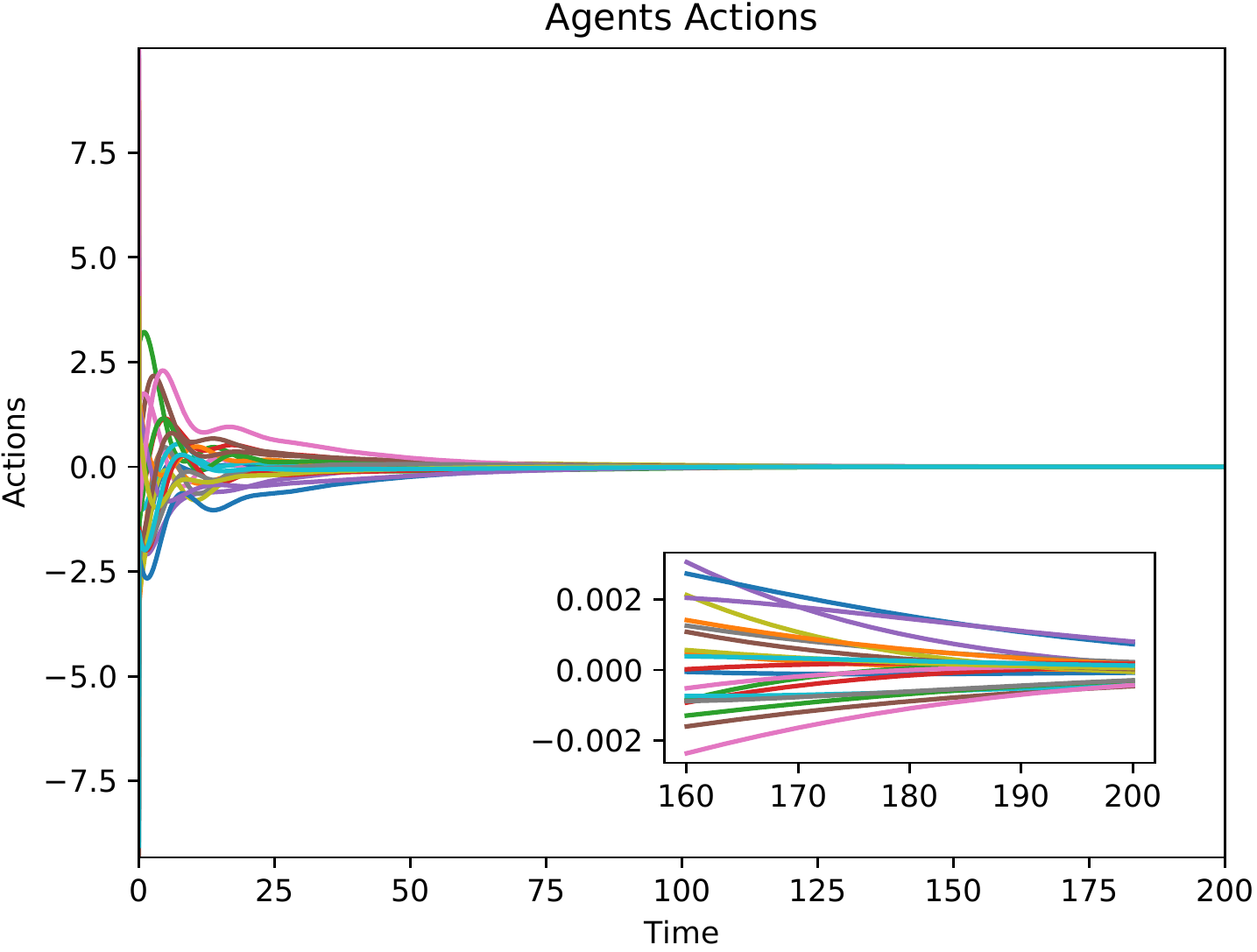}}
	\caption{Evolution of agents' actions}\label{fig:actions_nonlin}
\end{minipage}
\end{figure}


\section{Conclusion} \label{sec:conclusion}
In this paper, we considered monotone games and proposed a continuous-time dynamics constructed via passivity-based modification of a gradient-play scheme. We showed that in the full-decision information it converges to a Nash equilibrium in merely monotone games, for any positive parameter values. Under different assumptions we provided extensions to the partial-decision information case and extensions to hypomonotone games. Among future interesting problems we mention, extensions to directed communication graphs or, with adaptive gains, as well as to generalized Nash equilibrium problems.


\bibliographystyle{IEEEtran}
\bibliography{referencesCDC}

\section*{Appendix A} \label{sec:AppendixA}

\subsection*{Proof of Proposition \ref{lemma:resInvLip} }
(i): Notice that,
\begin{align*}
	\inp*{x-y}{(I+\lambda T)x - (I+\lambda T)y} &\geq \bracket{1-\lambda \mu}\snorm{x-y}
\end{align*}
by assumption $1> \lambda\mu$, therefore $(I+\lambda T)$ is strongly monotone. From Proposition 20.10 \cite{monoBookv2} the inverse of a monotone operator is monotone and therefore $\mathcal{J}_{\lambda T}$ is monotone. Additionally, since $(I+\lambda T)$ is strongly monotone, $\mathcal{J}_{\lambda T}$ is a single valued function.

(ii): Let $u = \mathcal{J}_{\lambda T}x$ and $v = \mathcal{J}_{\lambda T}y$, i.e., $x = (I+\lambda T)u$. Then,
\begin{align*}
	&\snorm{x-y} = \snorm{(I+\lambda T)u - (I+\lambda T)v} \\
	&\quad = \snorm{u-v} + \lambda^{2}\snorm{Tu-Tv} + 2\lambda\inp*{u-v}{Tu-Tv} \\
	&\quad \geq \bracket{1 + \frac{\lambda^{2}}{R^2}-2\lambda \mu}\snorm{u-v} \\
	&\quad = \bracket{\frac{R^2 + \lambda^{2}-2\lambda\mu R^{2}}{R^2}}\snorm{u-v}
\end{align*}
By assumption $\mu R < 1$ which implies that $R^2 + \lambda^{2}-2\lambda\mu R^{2} > R^2 + \lambda^{2}-2\lambda R = (\lambda - R)^{2} > 0$. If $\lambda = R$ then $R^2 + \lambda^{2}-2\lambda\mu R^{2} = 2R^{2}\bracket{1-\lambda \mu}$ and by assumption $\lambda \mu < 1$, therefore the numerator is always positive and
\begin{align*}
	\bracket{\frac{R^2}{R^2 + \lambda^{2} - 2\lambda \mu R^2}} \snorm{x-y} &\geq \snorm{\mathcal{J}_{\lambda T}x - \mathcal{J}_{\lambda T}y}
\end{align*} 

(iii): Assume that $R\geq \lambda$ and let $u = \mathcal{J}_{\lambda T}x$, $v = \mathcal{J}_{\lambda T}y$, and $c = \frac{R^{2}(1-\mu\lambda)}{R^{2} + \lambda^{2}-2\mu\lambda R^{2}}$. Then,
\begin{align}
	&c\snorm{x-y} = c\snorm{(I+\lambda T)u - (I+\lambda T)v} \notag \\
	&= c\snorm{u-v} + c\lambda^{2}\snorm{Tu - Tv} \notag + 2c\lambda\inp*{Tu - Tv}{u - v} \notag \\
	&\quad - \inp*{u-v}{(1+\lambda T)u-(1+\lambda T)v} \notag \\
	&\quad + \inp*{u-v}{(1+\lambda T)u-(1+\lambda T)v} \notag \\
	&= (c-1)\snorm{u-v} + c\lambda^{2}\snorm{Tu - Tv} \notag \\
	&\quad + \lambda (2c-1)\inp*{Tu - Tv}{u - v} \label{eqn:intStep} \\
	&\quad + \inp*{u-v}{(1+\lambda T)u-(1+\lambda T)v} \notag
\end{align}
Note that $2c-1 \geq 0$ for all $R\geq \lambda$ and $1> \mu\lambda$, so that
\begin{align*}
	&c\snorm{x-y} \geq c\bracket{\frac{R^2+\lambda^2-2\mu\lambda R^{2}}{R^2}}\snorm{u-v} \\
	&\quad - (1-\mu\lambda)\snorm{u-v} + \inp*{u-v}{(1+\lambda T)u-(1+\lambda T)v} \\
	&= \inp*{u-v}{(1+\lambda T)u-(1+\lambda T)v} \\
	&= \inp*{\mathcal{J}_{\lambda T}x-\mathcal{J}_{\lambda T}y}{x-y}
\end{align*}
Now assume that $\lambda \geq R$ and $c = \frac{R^{2}(1+ \frac{\lambda}{R})}{R^{2} + \lambda^{2} + 2\lambda R}$ then $2c-1 \leq 0$. Continuing from \eqref{eqn:intStep} and using the fact that $-\inp*{a}{b} \geq -\frac{1}{4s}\snorm{a}-s\snorm{b}$, yields
\begin{align*}
	&c\snorm{x-y} \geq (c-1)\snorm{u-v} + c\lambda^{2}\snorm{Tu - Tv} \\
	&\qquad + \lambda (2c-1)\bracket{\frac{R}{2}\snorm{Tu-Tv} + \frac{1}{2R}\snorm{u-v}} \\
	&\qquad + \inp*{u-v}{(1+\lambda T)u-(1+\lambda T)v} \\
	&\quad = c\bracket{\frac{R^2 + \lambda^2 + 2\lambda R}{R^2}}\snorm{u-v} -\bracket{1+\frac{\lambda}{R}}\snorm{u-v} \\
	&\qquad + \inp*{u-v}{(1+\lambda T)u-(1+\lambda T)v} \\
	&\quad = \inp*{\mathcal{J}_{\lambda T}x-\mathcal{J}_{\lambda T}y}{x-y}
\end{align*}

\subsection*{Proof of Proposition \ref{lemma:funcType} }
(i) From \cite{FP07} Prop 2.3.2 (c).
	
(ii) From \cite{FP07} Prop 2.9.25 (a).

(iii)
\begin{align*}
	\snorm{Tx-Ty} &= \snorm{\bracket{\int_{0}^{1}JT(x + t(y-x)) \partial t}(y-x)} \\
	&\leq \max_{z} \snorm{JT(z)}\snorm{x-y}\\
	&\leq L^{2} \snorm{x-y}
\end{align*}

(iv) Note that,
\begin{align*}
\snorm{Tx-Ty}
&= \snorm{\bracket{\int_{0}^{1}JT(x + t(y-x)) \partial t}(y-x)} \\
&\geq \min_{z} \snorm{\bracket{\int_{0}^{1}JT(z) \partial t}(y-x)} \\
&= \min_{z} (y-x)^{T}\bracket{JT^{T}(z)JT(z)}(y-x) \\
&\geq \frac{1}{R^{2}} \snorm{y-x}
\end{align*}

\subsection*{Proof of Lemma \ref{lemma:consTTS} }
First, note that   $\mathbf{F}(\mathbf{\bar{r}}^{\orthc}) =F(x^*)= \mathbf{0}_{n}$ for $\mathbf{\bar{r}}^{\orthc} = \mathbf{1}_{N}\otimes x^{*}$. Therefore, since  $\mathbf{h}$ \eqref{h_def} is the resolvent  of $\mathbf{1}_{N}\otimes \mathbf{F}$ on the consensus subspace,  and zeros of $\mathbf{1}_{N}\otimes \mathbf{F}$ are fixed points of the resolvent, (cf. Prop. 23.2, \cite{monoBookv2}), it follows that $\mathbf{\bar{r}}^{\orthc} = \mathbf{h}(\mathbf{\bar{r}}^{\orthc})$. Using this, along \eqref{eqn:consDynFI}, we can write
\begin{align}\label{r_term_0}
\hspace{-0.3cm}	\inp*{\mathbf{r}^{\orthc} - \mathbf{\bar{r}}^{\orthc}}{\mathbf{\dot{r}}^{\orthc}}
&	=  \alpha\inp*{\mathbf{r}^{\orthc} - \mathbf{\bar{r}}^{\orthc}}{\mathbf{z}^{\orthc}} - \alpha\snorm{\mathbf{r}^{\orthc} - \mathbf{\bar{r}}^{\orthc}} \\
&		+ \alpha\inp*{\mathbf{r}^{\orthc} - \mathbf{\bar{r}}^{\orthc}}{\mathbf{h}(\mathbf{r}^{\orthc}) - \mathbf{h}(\mathbf{\bar{r}}^{\orthc})}. \notag
\end{align}
To bound the last term we use Lemma \ref{lemma:resInvLip} as follows. For any $r\in  \reals^n$ let  $h(r):=\mathcal{J}_{\frac{1}{\beta N}{F}}(r)$ the resolvent of $F$. Using 
$F(h(r)) =\mathbf{F}(\mathbf{1}_{N}\otimes  h(r))$, we can write  $\mathbf{1}_{N}\otimes\big(\text{Id}+ \frac{1}{\beta N}F\big)h(r)= (\text{Id} + \frac{1}{\beta N}\mathbf{1}_{N}\otimes\mathbf{F}  )(\mathbf{1}_{N}\otimes  h(r))$. Using $\big (\text{Id}+ \frac{1}{\beta N}F \big)h(r)=r$ and  \eqref{h_def}, this is equivalent to  $\mathbf{h}(\mathbf{1}_{N}\otimes r)  
=  \mathbf{1}_{N}\otimes h(r)$. As   $h$  is the resolvent  of $F$,  under Assumption  \ref{asmp:monotone} and \ref{asmp:inverse}, we apply Lemma \ref{lemma:resInvLip} to $F$ with $\lambda = \frac{1}{\beta N}$. Therefore, since for any  $\mathbf{r}^{\orthc} = \mathbf{1}_{N}\otimes r$,  $\mathbf{h}(\mathbf{r}^{\orthc}) =\mathbf{1}_{N}\otimes h(r)$ the bounds from Lemma \ref{lemma:resInvLip} (ii) and (iii) hold, and it follows that the same bounds hold for $\mathbf{h}$. Using Lemma \ref{lemma:resInvLip} (iii) in the last term of \eqref{r_term_0} yields,  
\begin{align}\label{r_term}
\begin{split}
\inp*{\mathbf{r}^{\orthc} - \mathbf{\bar{r}}^{\orthc}}{\mathbf{\dot{r}}^{\orthc}} &\leq \alpha \norm{\mathbf{r}^{\orthc} - \mathbf{\bar{r}}^{\orthc}}\norm{\mathbf{z}^{\orthc}} \\
	&\qquad - \alpha \bracket{1-{\kappa_{\mathcal{J}}}}  \snorm{\mathbf{r}^{\orthc} - \mathbf{\bar{r}}^{\orthc}}
\end{split}
\end{align}
Similarly, using \eqref{eqn:consDynFI}, we can write, 
\begin{align*}
\inp*{\mathbf{z}^{\orthc}}{\mathbf{\dot{z}}^{\orthc}}  &= -\frac{1}{N}\inp*{\mathbf{z}^{\orthc}}{\mathbf{1}_N\otimes\mathbf{F}(\mathbf{z}^{\orthc} + \mathbf{h}(\mathbf{r}^{\orthc}))} \\
	&\qquad - \beta \inp*{\mathbf{z}^{\orthc}}{\mathbf{z}^{\orthc} + \mathbf{h}(\mathbf{r}^{\orthc}) - \mathbf{r}^{\orthc}} \\
	&\qquad - \alpha\inp*{\mathbf{z}^{\orthc}}{\left[\frac{\partial \mathbf{h}}{\partial \mathbf{r}}\right](\mathbf{z}^{\orthc} + \mathbf{h}(\mathbf{r}^{\orthc}) - \mathbf{r}^{\orthc})}
\end{align*}

Substituting $\mathbf{r}^{\orthc} = (\text{Id} + \mathbf{1}_{N}\otimes \frac{1}{\beta N}\mathbf{F})\mathbf{h}(\mathbf{r}^{\orthc})$ (cf. \eqref{h_def}) in the middle term and combining terms yields, 
\begin{align*}
	&\inp*{\mathbf{z}^{\orthc}}{\mathbf{\dot{z}}^{\orthc}} \\
& = -\frac{1}{N}\inp*{\mathbf{z}^{\orthc}}{\mathbf{1}_{N}\otimes\mathbf{F}(\mathbf{z}^{\orthc} + \mathbf{h}(\mathbf{r}^{\orthc}))  - \mathbf{1}_{N}\otimes\mathbf{F}(\mathbf{h}(\mathbf{r}^{\orthc}))} \\
	&\qquad - \beta\snorm{\mathbf{z}^{\orthc}} - \alpha \inp*{\mathbf{z}^{\orthc}}{\left[\frac{\partial \mathbf{h}}{\partial \mathbf{r}}\right](\mathbf{z}^{\orthc} + \mathbf{h}(\mathbf{r}^{\orthc}) - \mathbf{r}^{\orthc})}
\end{align*}
The first term is non-negative since $\mathbf{z}^{\orthc}$, $\mathbf{z}^{\orthc} + \mathbf{h}(\mathbf{r}^{\orthc})$ and $\mathbf{h}(\mathbf{r}^{\orthc})$ are on the consensus subspace and $\mathbf{1}_{N}\otimes \mathbf{F}$ evaluates to just $\mathbf{1}_{N}\otimes F$, which is $\mu$-hypomonotone by Assumption \ref{asmp:inverse}. Adding and subtracting  $\mathbf{\bar{r}}^{\orthc} = \mathbf{h}(\mathbf{\bar{r}}^{\orthc})$ in the last term, we can then write 
\begin{align*}
\inp*{\mathbf{z}^{\orthc}}{\mathbf{\dot{z}}^{\orthc}} &\leq -\bracket{\beta - \frac{\mu}{N}}\snorm{\mathbf{z}^{\orthc}} - \alpha\inp*{\mathbf{z}^{\orthc}}{\left[\frac{\partial \mathbf{h}}{\partial \mathbf{r}}\right ]\mathbf{z}^{\orthc}} 
\\
	&\quad + \alpha \norm{\mathbf{z}^{\orthc}}\norm{\frac{\partial \mathbf{h}}{\partial \mathbf{r}}}\norm{\mathbf{h}(\mathbf{r}^{\orthc}) - \mathbf{h}(\mathbf{\bar{r}}^{\orthc})} \\
	&\quad + \alpha\norm{\mathbf{z}^{\orthc}}\norm{\frac{\partial \mathbf{h}}{\partial \mathbf{r}}}\norm{\mathbf{r}^{\orthc} - \mathbf{\bar{r}}^{\orthc}}.
\end{align*}
The second term is non-positive since  $\mathbf{h}$ is monotone by Lemma \ref{lemma:resInvLip} (i) and  $\frac{\partial \mathbf{h}}{\partial \mathbf{r}}$ is positive semidefinite (cf. Proposition 2.3.2 \cite{FP07}). Using   $\norm{\mathbf{h}(\mathbf{r}^{\orthc}) - \mathbf{h}(\mathbf{\bar{r}}^{\orthc})} \leq L_{\mathcal{J}} \norm{\mathbf{r}^{\orthc} - \mathbf{\bar{r}}^{\orthc}}$ and $\norm{\frac{\partial \mathbf{h}}{\partial \mathbf{r}}} \leq L_{\mathcal{J}}$ from Lemma \ref{lemma:resInvLip}(ii), yields, 
\begin{align}
\begin{split}
\inp*{\mathbf{z}^{\orthc}}{\mathbf{\dot{z}}^{\orthc}} &\leq  - \bracket{\beta-\frac{\mu}{N}}\snorm{\mathbf{z}^{\orthc}} \\
	&\quad + \alpha\bracket{L_{\mathcal{J}} + L_{\mathcal{J}}^{2}}\norm{\mathbf{z}^{\orthc}}\norm{\mathbf{r}^{\orthc} - \mathbf{\bar{r}}^{\orthc}}
\end{split}\label{z_term}
\end{align}
Finally,  for $V^{\orthc}$ as in the lemma, using the bounds in \eqref{r_term}, \eqref{z_term},  along the solution of \eqref{eqn:consDynFI},  we can write $\dot{V}^{\orthc}(\mathbf{z}^{\orthc}, \mathbf{r}^{\orthc}) \leq -\omega^{T}\Phi \omega$, where   $\Phi $ is as in \eqref{eqn:lyapNSD} and $\omega = \bracket{\norm{\mathbf{r}^{\orthc} - \mathbf{\bar{r}}^{\orthc}}, \norm{\mathbf{z}^{\orthc}}}$. It can be easily seen  that for any given $d\in(0,1)$ and $\alpha$ as in the lemma, $\Phi$ is positive definite.

\subsection*{Proof of Lemma \ref{lemma:full_hypo} }
Note that if we start with \eqref{eqn:proposedAlg} and do the change of coordinates $z = x-\mathcal{J}_{\frac{1}{\beta}F}r$ we get \eqref{eqn:consDynFI} but with $\mathbf{r}^{\orthc}$ replaced with $r$, $\mathbf{x}^{\orthc}$ replaced  with $x$, $\mathbf{z}^{\orthc}$ with $z = x-\mathcal{J}_{\frac{1}{\beta}F}r$, and $\frac{1}{N}\mathbf{1}_{N}\otimes \mathbf{F}$ replaced  with $F$. Therefore, following the same argument as Lemma \ref{lemma:consTTS} we can construct a Lyapunov function that shows that $x^{*}$ is asymptotically stable.

\subsection*{Proof of Lemma \ref{lemma:eigRelation} }
Let $v_{i} = (x_{i},y_{i})$ be the $i^{th}$ eigenvector of $M$ \eqref{eqn:DynLT} then,
\begin{align*}
	\begin{bmatrix}
		-A-\beta I & \beta I \\
		\alpha I & -\alpha I
	\end{bmatrix}\begin{bmatrix}
	x_{i} \\
	y_{i}
	\end{bmatrix} &= \lambda_{i} \begin{bmatrix}
	x_{i} \\
	y_{i}
	\end{bmatrix}
\end{align*}
The second row implies that $x_{i} = \frac{\alpha + \lambda_{i}}{\alpha}y_{i}$. Substituting this into the first row, yields
\begin{align*}
	Ay_{i} &= -\frac{\lambda_{i}(\alpha + \beta + \lambda_{i})}{\alpha + \lambda_{i}}y_{i}
\end{align*}
This equation can only hold true if $y_{i}$ is an eigenvector for $A$. With $\rho_{i}$ is the corresponding eigenvalue. Therefore,
\begin{align*}
	\rho_{i} &= -\frac{\lambda_{i}(\alpha + \beta + \lambda_{i})}{\alpha + \lambda_{i}}
\end{align*}
and solving for the roots of the quadratic in $\lambda_{i}$ gives \eqref{eqn:eigMap}.

\subsection*{Proof of Lemma \ref{thm:stableA}}
(i) From Lemma \ref{lemma:eigRelation} we see that the characteristic polynomial is $\mathcal{C}_{i} \defeq \lambda_{i}^{2} + (\alpha + \beta + \rho_{i})\lambda_{i} + \rho_{i}$ and when $\rho_{i} = 0$ we immediately get our result.

(ii) We need to show that the real part of the roots of $\mathcal{C}_{i}$ must be less than 0. From \cite{GenRouth} we know that the roots of a complex coefficient polynomial are in the left half plane if the roots of, 
\begin{align*}
	\mathcal{C}_{i}^{*}\mathcal{C}_{i} &= \lambda_{i}^{4} + 2(\alpha + \beta + r_{i})\lambda_{i}^{3} \\
	&\quad + \bracket{(\alpha+\beta +r_{i})^{2} + k_{i}^{2} + 2\alpha r_{i}}\lambda_{i}^{2} \\
	&\quad + 2\bracket{r_{i}(\alpha+\beta+r_{i})+k_{i}^2}\lambda_{i} + \alpha^{2}\bracket{r_{i}^{2}+k_{i}^{2}}
\end{align*}
are in the left half plane. The Routh array for $\mathcal{C}_{i}^{*}\mathcal{C}_{i}$ is,
\begin{equation*}
\begin{array}{c|c|c}
  1 & \bracket{(\alpha+\beta +r_{i})^{2} + k_{i}^{2} + 2\alpha r_{i}} & \alpha^{2}\bracket{r_{i}^{2}+k_{i}^{2}} \\ 
  \hline
  2(\alpha + \beta + r_{i}) & 2\bracket{\alpha r_{i}(\alpha+\beta+r_{i})+\alpha k_{i}^2} & 0\\
  \hline 
  T_1 & \alpha^{2}\bracket{r_{i}^{2}+k_{i}^{2}} & 0 \\
  \hline
  T_2 & 0 & 0 \\
  \hline 
  \alpha^{2}\bracket{r_{i}^{2}+k_{i}^{2}} & 0 & 0
 \end{array}
\end{equation*}
where
\begin{align*}
T_1 &= \left[ \underbrace{(\alpha +\beta + r_{i})^2}_{>0} + \underbrace{\alpha r_{i}}_{\geq 0} + \underbrace{\frac{\beta + r_{i}}{\alpha+\beta + r_{i}}}_{>0}\underbrace{k_{i}^2}_{\geq 0}\right] > 0 \\
T_2 &= \frac{2\alpha}{T_1}\left[ \underbrace{r_{i}(\alpha+\beta + r_{i})^3}_{\geq 0} + \underbrace{\frac{\beta + r_{i}}{\alpha + \beta + r_{i}}k_{i}^4}_{\geq 0} \right. \\
	&\qquad \qquad \qquad \left. + \underbrace{(\alpha+\beta+r_{i})}_{\geq 0}\underbrace{(\beta+2r_{i})}_{\geq 0} \underbrace{k_{i}^2}_{\geq 0} \right] > 0
\end{align*}

If we show that all elements in the left column in the Routh array are all positive then the roots of $\mathcal{C}_{i}$ are less than $0$. The term $2(\alpha + \beta + r_{i})$, $\alpha^{2}\bracket{r_{i}^2+k_{i}^2}$, and $T_1$ are positive. Either $r_{i}\neq 0$ or $k_{i}\neq 0$, therefore one of the terms in $T_{2}$ will be strictly positive making $T_{2} > 0$. Therefore, $\lambda_{i}$ has real part less than $0$.

(iii) The term $\alpha^{2}\bracket{r_{i}^2+k_{i}^2}$ is always positive. By assumption, $\alpha > 0$ and $\beta + r_{i} > 0$, therefore the term $2(\alpha+\beta+r_{i})$ is positive. For the $T_{1}$ term, let $\epsilon = \beta + r_{i}>0$ then,
\begin{align*}
	T_1 & = \left[(\epsilon+\alpha)^2 + \alpha r_{i} + \frac{\epsilon}{\epsilon+\alpha}k_{i}^2\right]
\end{align*}
Multiplying $T_{1}$ by $\epsilon + a > 0$ gives the condition,
\begin{align*}
	0 &< \left[(\epsilon+\alpha)^3 + \alpha(\epsilon+\alpha)r_{i} + \epsilon k_{i}^2\right] \\
	0 &< (\epsilon+\alpha)^3 + [\alpha + \epsilon - \epsilon](\epsilon+\alpha)r_{i} + \epsilon k_{i}^2 \\
	0 &< (\epsilon+\alpha)^3 - \epsilon (\epsilon + \alpha)r_{i} + \epsilon k_{i}^2  + (\epsilon + \alpha)^2r_{i}
\end{align*}
From the upper bound assumption on $\alpha$ we see that $\alpha$ satisfies, $\alpha < -\epsilon + \sqrt{\frac{-\epsilon k_{i}^2}{r_{i}}} \implies \epsilon k_{i}^2 + (\alpha+\epsilon)^2r_{i} > 0$. Therefore the condition is always satisfied. Note that as $\alpha \to 0$ that the condition becomes the assumption for the upper bound of $\beta$. For the $T_{2}$ term,
\begin{align*}
	T_2 &= \frac{2\alpha}{T_1}\bracket{ r_{i}(\alpha+\epsilon)^3 + \frac{\epsilon}{\alpha + \epsilon}k_{i}^4 + (\alpha+\epsilon)(\epsilon +r_{i}) k_{i}^2 } \\
	&= \frac{2\alpha}{(\epsilon+\alpha)T_1}\bracket{r_{i}(\epsilon+\alpha)^4 + \epsilon k_{i}^4 + (\alpha+\epsilon)^2(\epsilon + r_{i}) k_{i}^2 }
\end{align*}
Since $\frac{2}{(\epsilon+\alpha)T_1} > 0$ the condition for $T_{2}>0$ is,
\begin{equation*}
	r_{i}x^2 + (\epsilon +r_{i}) k_{i}^2x + \epsilon k_{i}^4 > 0
\end{equation*}
where $x = (\epsilon + \alpha)^2$. The roots of this equation are,
\begin{align*}
	x &= \frac{1}{2r_{i}}[-(\epsilon + r_{i})k_{i}^2 \pm \sqrt{(\epsilon +r_{i})^2k_{i}^4 - 4\epsilon r_{i}k_{i}^4 } \\
	&= \frac{1}{2r_{i}}[-(\epsilon+r_{i})k_{i}^2 \pm (\epsilon - r_{i})k_{i}^2 ] = \frac{-\epsilon k_{i}^2}{r_{i}} \text{ or }  -k_{i}^2
\end{align*}
Therefore, $x\in (-k_{i}^{2},\frac{-\epsilon k_{i}^2}{r_{i}})$ for $T_{2}>0$, but $x=(\epsilon+\alpha)^{2}$ so $x = (\epsilon+\alpha)^{2}\in (0,\frac{-\epsilon k_{i}^2}{r_{i}})$ which implies,
\begin{align*}
	0 < \alpha <  -\bracket{\beta+r_{i}} + \sqrt{\frac{-(\beta+r_{i})k_{i}^2}{r_{i}}}
\end{align*}

\subsection*{Proof of Theorem \ref{thm:distQuad} }
After performing a change of coordinates as in \eqref{eqn:DynLT} and a decomposition as in the nonlinear case, the dynamics \eqref{eqn:distProposedAlg} can be written as,
\begin{align}
	\mathbf{\dot{x}^{\orthc}} &= \Pi_{\orthc}\bracket{\mathcal{R}^{T}\mathbf{A}\mathbf{x} - \beta(\mathbf{x} - \mathbf{r}) - c\mathbf{L}\mathbf{x}} \label{eqn:x_par}\\
	&= \Pi_{\orthc}\mathcal{R}^{T}\mathbf{A}(\mathbf{x}^{\orthc} + \mathbf{x}^{\orth}) - \beta(\mathbf{x}^{\orthc} - \mathbf{r}^{\orthc}) \notag \\
	\mathbf{\dot{x}^{\orth}} &= (I-\Pi_{\orthc})\bracket{\mathcal{R}^{T}\mathbf{A}\mathbf{x} - \beta(\mathbf{x} - \mathbf{r}) - c\mathbf{L}\mathbf{x}} \notag \\
	&= (I-\Pi_{\orthc})\mathcal{R}^{T}\mathbf{A}(\mathbf{x}^{\orthc} + \mathbf{x}^{\orth}) - \beta(\mathbf{x}^{\orth} - \mathbf{r}^{\orth}) - c\mathbf{L}\mathbf{x}^{\orth} \notag \\
	\mathbf{\dot{r}}^{\orthc} &= \alpha (\mathbf{x}^{\orthc} - \mathbf{r}^{\orthc}) \label{eqn:r_par}\\
	\mathbf{\dot{r}}^{\orth} &= \alpha (\mathbf{x}^{\orth} - \mathbf{r}^{\orth}) \notag
\end{align}
Let $\mathbf{w} = \mathbf{w}^{\orthc} + \mathbf{w}^{\orth}$, $\mathbf{w}^{\orthc} = (\mathbf{x}^{\orthc},\mathbf{r}^{\orthc})$, $\mathbf{w}^{\orth} = (\mathbf{x}^{\orth},\mathbf{r}^{\orth})$, and $(\mathbf{w}^{i})^{\orthc} = ((\mathbf{x}^{i})^{\orthc}, (\mathbf{r}^{i})^{\orthc})$. The matrix $\tilde{M}$ has the same structure as $M$ (some terms scaled). From Lemma \ref{thm:stableA} we know that $\tilde{M}$, for the $\alpha$ and $\beta$ satisfying the assumptions in the theorem, has all its eigenvalues  with real part less than $0$ and therefore there exists a $P$ satisfying the Lyapunov equation. Consider the following Lyapunov function, 
\begin{align}
V(\mathbf{w})\! &= \frac{1}{2}\snorm{\mathbf{x}^{\orth}} \! + \! \frac{\beta}{2\alpha}\snorm{\mathbf{r}^{\orth}} \! + \! \sum_{i\in\mathcal{N}} \snorm{(\mathbf{w}^{i})^{\orthc}-w^{*}}_{P} \label{eqn:cand_lyap}
\end{align}
For the first two terms in \eqref{eqn:cand_lyap},
\begin{align*}
&\frac{d}{d t}\bracket{\frac{1}{2}\snorm{\mathbf{x}^{\orth}} + \frac{\beta}{2\alpha}\snorm{\mathbf{r}^{\orth}}} \\
	& = (\mathbf{x}^{\orth})^{T}\bracket{(I-\Pi_{\orthc})\mathcal{R}^{T}\mathbf{A}(\mathbf{x}^{\orthc}+\mathbf{x}^{\orth}) - cL\mathbf{x}^{\orth}} \\
	&\qquad - (\mathbf{w}^{\orth})^{T}\begin{bmatrix}
		\beta I & -\beta I \\
		-\beta I & \beta I
	\end{bmatrix}\mathbf{w}^{\orth}
\end{align*}
Since the last term is equal to $-\beta(\mathbf{x}^{\orth} - \mathbf{r}^{\orth})^{2} \leq 0$, therefore
\begin{align*}
&\frac{d}{d t}\bracket{\frac{1}{2}\snorm{\mathbf{x}^{\orth}} + \frac{\beta}{2\alpha}\snorm{\mathbf{r}^{\orth}}} \\
&\leq (\mathbf{x}^{\orth})^{T}(I-\Pi_{\orthc})\mathcal{R}^{T}\mathbf{A}(\mathbf{x}^{\orthc}+\mathbf{x}^{\orth}) - c\lambda_{2}(L)\snorm{\mathbf{x}^{\orth}}
\end{align*}
Using $\norm{(I-\Pi_{\orthc})\mathcal{R}^{T}\mathbf{A}}= \sqrt{\frac{N-1}{N}}\norm{\mathbf{A}}$, $\norm{\mathbf{A}} \leq L_{\mathbf{A}}$ and $\mathcal{R}^{T}\mathbf{A}\mathbf{x}^{*} = 0$ yields,
\begin{align}
&\frac{d}{d t}\bracket{\frac{1}{2}\snorm{\mathbf{x}^{\orth}} + \frac{\beta}{2\alpha}\snorm{\mathbf{r}^{\orth}}} \notag \\
&\leq \sqrt{\frac{N-1}{N}}L_{\mathbf{A}}\bracket{ \snorm{\mathbf{x}^{\orth}} + \norm{\mathbf{x}^{\orth}}\norm{\mathbf{x}^{\orthc}-x^{*}}} \notag \\
&\qquad - c\lambda_{2}(L)\snorm{\mathbf{x}^{\orth}} \notag \\
	&\leq \sqrt{\frac{N-1}{N}}L_{\mathbf{A}}\bracket{ \snorm{\mathbf{x}^{\orth}} + \norm{\mathbf{x}^{\orth}}\norm{\mathbf{w}^{\orthc}-w^{*}}} \label{eqn:lyp_bound_1} \\
&\qquad - c\lambda_{2}(L)\snorm{\mathbf{x}^{\orth}} \notag
\end{align}

Note that $\mathbf{w}^{\orthc}$ can be written as $\mathbf{w}^{\orthc} = 1\otimes w^{\orthc}$ and $(\mathbf{w}^{i})^{\orthc}=(\mathbf{w}^{j})^{\orthc} = w^{\orthc}$. For the third term in \eqref{eqn:cand_lyap}, along the solution of \eqref{eqn:x_par} and \eqref{eqn:r_par},
\begin{align*}
	&\frac{d}{d t} \sum_{i\in\mathcal{N}} \snorm{(\mathbf{w}^{i})^{\orthc}-w^{*}}_{P} \\
	&\quad = \sum_{i\in\mathcal{N}} ((\mathbf{w}^{i})^{\orthc}-w^{*})^{T}(P\tilde{M}+\tilde{M}^{T}P)((\mathbf{w}^{i})^{\orthc}-w^{*}) \\
	&\qquad + ((\mathbf{w}^{i})^{\orthc}-w^{*})^{T}(P+P^{T})Q\mathbf{w}^{\orth}
\end{align*}
where $Q = \begin{bmatrix}
	\frac{-1}{N}\mathbf{A} & 0 \\
	0 & 0
	\end{bmatrix}$.
\begin{align*}
	&\quad = -\snorm{\mathbf{w}^{\orthc}-\mathbf{w}^{*}} + \sum_{i\in\mathcal{N}}((\mathbf{w}^{i})^{\orthc}-w^{*})^{T}(P+P^{T})Q\mathbf{w}^{\orth} \\
	&\quad = -\snorm{\mathbf{w}^{\orthc}-\mathbf{w}^{*}} + (\mathbf{w}^{\orthc}-w^{*})^{T}[I\otimes (P+P^{T})][\mathbf{1}\otimes Q]\mathbf{w}^{\orth}
\end{align*}
Using $\norm{P} = p$ and $\norm{\frac{1}{N}\mathbf{1}\otimes \mathbf{A}} \leq \frac{1}{\sqrt{N}}L_{\mathbf{A}}$,
\begin{align}
\begin{split}
	&\frac{d}{d t} \sum_{i\in\mathcal{N}} \snorm{(\mathbf{w}^{i})^{\orthc}-w^{*}}_{V} \\
	&\quad \leq -\snorm{\mathbf{w}^{\orthc}-\mathbf{w}^{*}}  + \frac{2p L_{\mathbf{A}}}{\sqrt{N}}\norm{\mathbf{w}^{\orthc}-\mathbf{w}^{*}}\norm{\mathbf{x}^{\orth}}
\end{split} \label{eqn:lyp_bound_2}
\end{align}

Therefore, from \eqref{eqn:lyp_bound_1} and \eqref{eqn:lyp_bound_2} the Lyapunov function $V$, \eqref{eqn:cand_lyap}, satisfies 
\begin{align*}
	\frac{d}{d t}V(\mathbf{w}) &\leq 						\varpi^{T}\begin{bmatrix}
	-1 & \frac{L_{\mathbf{A}}}{2\sqrt{N}}\bracket{2p + \sqrt{N-1}} \\
	* & -c\lambda_{2}(L) + \sqrt{\frac{N-1}{N}}L_{\mathbf{A}}
	\end{bmatrix} \varpi
\end{align*}
where $\varpi = col(\norm{\mathbf{w}^{\orthc}-\mathbf{w}^{*}},\norm{\mathbf{x}^{\orth}})$. Under \eqref{eqn:GraphCondition} the matrix is negative definite and using LaSalle's Invariance Principle \cite{nonlinear} concludes the proof.

\end{document}